\documentclass[10pt]{article}
\usepackage{amssymb}
\usepackage{graphicx}
\usepackage{xcolor} 
\usepackage{tensor}
\usepackage{fullpage} 
\usepackage{amsmath}
\usepackage{amsthm}
\usepackage{verbatim}
\DeclareFontFamily{U}{mathx}{\hyphenchar\font45}
\DeclareFontShape{U}{mathx}{m}{n}{
	<5> <6> <7> <8> <9> <10>
	<10.95> <12> <14.4> <17.28> <20.74> <24.88>
	mathx10
}{}
\DeclareSymbolFont{mathx}{U}{mathx}{m}{n}
\DeclareFontSubstitution{U}{mathx}{m}{n}
\DeclareMathAccent{\widecheck}{0}{mathx}{"71}
\DeclareMathAccent{\wideparen}{0}{mathx}{"75}

\usepackage{authblk}
\usepackage{enumitem}
\setlist[enumerate]{leftmargin=1.5em}
\setlist[itemize]{leftmargin=1.5em}

\newtheorem{theorem}{Theorem}[section]

\newtheorem{lemma}[theorem]{Lemma}
\newtheorem{remark}[theorem]{Remark}

\newtheorem{proposition}[theorem]{Proposition}

\numberwithin{equation}{section}

\newcommand{\norm}[1]{{\left\Vert #1 \right\Vert}}
\newcommand{\normif}[1]{{\left\Vert #1 \right\Vert}_{L^{\infty}}}

\newcommand{\normb}[1]{{\left\Vert #1 \right\Vert}_{L^2}}

\newcommand{\normd}[1]{{\left\Vert #1 \right\Vert}_{L^4}}

\newcommand{\normifr}[1]{{\left\Vert #1 \right\Vert}_{L^{\infty}\left (\mathbb{R}^2\right  )}}

\newcommand{\normbr}[1]{{\left\Vert #1 \right\Vert}_{L^2\left (\mathbb{R}^2\right  )}}

\newcommand{\normdr}[1]{{\left\Vert #1 \right\Vert}_{L^4\left (\mathbb{R}^2\right  )}}

\newcommand{\normp}[1]{{\left\Vert #1 \right\Vert}_{L^p}}
\newcommand{\normq}[1]{{\left\Vert #1 \right\Vert}_{L^q}}

\newcommand{\normqr}[1]{{\left\Vert #1 \right\Vert}_{L^q\left (\mathbb{R}^2\right  )}}
\newcommand{\normrr}[1]{{\left\Vert #1 \right\Vert}_{L^r\left (\mathbb{R}^2\right  )}}

\newcommand{\normh}[1]{{\left\Vert #1 \right\Vert}_{H^m}}
\newcommand{\normhh}[1]{{\left\Vert #1 \right\Vert}_{H^{m+1}}}
\newcommand{\normhs}[1]{{\left\Vert #1 \right\Vert}_{H^s}}

\newcommand{\RN}[1]{%
  \textup{\uppercase\expandafter{\romannumeral#1}}%
}

\allowdisplaybreaks

\title{Global well-posedness for a two-dimensional Keller-Segel-Euler system of consumption type}
\author[1]{Jungkyoung Na\thanks{Corresponding author. \newline  E-mail address: njkyoung09@snu.ac.kr}}

\begin{document}
\maketitle

\begin{abstract}
We consider the Cauchy problem for the Keller-Segel system of consumption type coupled with the incompressible Euler equations in $\mathbb{R}^2$. This coupled system describes a biological phenomenon in which aerobic bacteria living in slightly viscous fluids (such as water) move towards a higher oxygen concentration to survive. We firstly prove the local existence of smooth solutions for arbitrary smooth initial data. Then we show that these smooth solutions can be extended globally if the initial density of oxygen is sufficiently small. The main ingredient in the proof is the $W^{1,q}$-energy estimate $(q>2)$ motivated by the partially inviscid two-dimensional Boussinesq system. Our result improves the well-known global well-posedness of the two-dimensional Keller-Segel system of consumption type coupled with the incompressible Navier-Stokes equations.
\end{abstract}


\section{Introduction}
Chemotaxis indicates the movement of biological cells or species toward a higher (or lower) concentration of some chemical substance. 
Since it describes many significant natural phenomena, its mathematical as well as biological analysis has attracted a lot of interest. Its first appearance in mathematical literature was in \cite{P53} by Patlak and in \cite{KS70, KS71} by Keller-Segel. One simplified version of their models can be written as
\begin{equation}\label{c-KS}\tag{c-KS}
    \left\{
    \begin{aligned}
    &\partial_{t} \rho = \Delta \rho - \nabla \cdot \left (\rho \nabla c\right),\\
    &\tau\partial_{t} c = \Delta c - c +\rho,
    \end{aligned}
    \right.
\end{equation}
where $\tau\ge0$ is a constant. 
Here $\rho$ and $c$ denote cell (or species) density and concentration of chemical substances, respectively. It is well known that solutions concentrate finite mass in a measure zero region within some finite time when the mass of initial data is large enough (see \cite{JL92, HV97, Benoit, W13} and references therein). The finite-time singularity formation is mainly due to the positive term $\rho$ on the right-hand side of the equation of $c$. This positive term represents the cell's contribution to producing the chemicals.

After their modeling, a variety of modified models have been proposed and analyzed to reflect natural phenomena more precisely. In particular, since chemotactic mechanism commonly occurs within fluid, Tuval et al. derived the following system based on experiments:
\begin{equation}\label{KS-N}\tag{KS-NS}
    \left\{
    \begin{aligned}
    &\partial_{t} \rho + u \cdot \nabla \rho = \nu_{\rho}\Delta \rho - \nabla \cdot \left (\chi(c)\rho \nabla c\right),\\
    &\partial_{t} c + u \cdot \nabla c = \nu_{c}\Delta c - k(c) \rho,\\
    &\partial_{t} u + u \cdot \nabla u + \nabla p = \nu_{u}\Delta u +\rho \nabla \phi,\\
    &\nabla \cdot u = 0,
    \end{aligned}
    \right.
\end{equation}
where $\nu_{\rho}, \nu_{c}, \nu_{u}>0$ represent diffusion coefficients.
Here $u$ and $p$ denote the fluid velocity and its pressure, respectively. $\chi(\cdot),k(\cdot):[0,\infty) \rightarrow [0,\infty)$ are smooth functions representing the chemotactic sensitivity and  consumption rate of chemical substances, respectively. $\phi$ is a smooth and time-independent function denoting the gravitational potential. The main difference from \eqref{c-KS} is twofold. The equation of $c$ has the negative term $-k(c)\rho$, which means the cells consume the chemical substances after movement. This kind of mechanism is called `consumption type'. 
In this consumption type, the finite-time blow-up as in \eqref{c-KS} does not happen.
Moreover, to reflect the roles of fluid, they combined $(\rho,c)$ system with the incompressible Navier-Stokes equation.

In this paper, we consider a Cauchy problem of the system coupled with the incompressible Euler equations (rather than Navier-Stokes equations):
\begin{equation}\label{KS-E}\tag{KS-E}
    \left\{
    \begin{aligned}
    &\partial_{t} \rho + u \cdot \nabla \rho = \nu_{\rho}\Delta \rho - \nabla \cdot \left (\chi(c)\rho \nabla c\right),\\
    &\partial_{t} c + u \cdot \nabla c = \nu_{c}\Delta c - k(c) \rho,\\
    &\partial_{t} u + u \cdot \nabla u + \nabla p = \rho \nabla \phi,\\
    &\nabla \cdot u = 0, \\
    &\rho(t=0)=\rho_0,\;\, c(t=0)=c_0,\;\, u(t=0)=u_0,\\
    \end{aligned}
    \right.
\end{equation}
in $\mathbb{R}^2 \times \left (0,T\right)$ for some $T>0$.
Specifically, the domains and codomains of functions in \eqref{KS-E} are as follows:
$\rho(x,t),\,c(x,t):\mathbb{R}^2 \times \left (0,T\right)\rightarrow [0,\infty)$, $u(x,t):\mathbb{R}^2 \times \left (0,T\right)\rightarrow \mathbb{R}^2$, $p(x,t):\mathbb{R}^2 \times \left (0,T\right)\rightarrow \mathbb{R}$, and $\phi:\mathbb{R}^2 \rightarrow \mathbb{R}$. The aim of this note is to show global well-posedness of \eqref{KS-E}. To effectively discuss the motivation for adopting the Euler equations and the significant meaning of our results, we firstly review some related previous results.

\subsection{Previous works}
\textbf{Fully viscous case.}
There have been a lot of studies focusing on the global regularity of \eqref{KS-N}. 
Global existence of smooth solutions in a two-dimensional bounded domain
with no-flux boundary conditions for both $\rho$ and $c$ and no-slip boundary conditions for $u$
was established for large initial data in \cite{W12} under the assumptions that
\begin{equation*}
    \left(\frac{k}{\chi}\right)'>0, \quad \left(\frac{k}{\chi}\right)''\le 0, \quad \left(\chi \cdot k\right)'\ge0
\end{equation*}
(see also \cite{W14} for stabilization of solutions to the equilibrium state under the same conditions).
This global well-posedness result was extended to the case covering slightly more general conditions of $\chi$ and $k$ in \cite{AKY21}. 
Recently, under different boundary conditions which can be considered more realistic than no-flux and no-slip conditions, the authors of \cite{WWX22} proved the global existence of smooth solutions in a two-dimensional bounded domain when initial $\rho$ and $c$ is sufficiently small and $\chi$ and $k$ are defined by $\chi(c)\equiv 1$ and $k(c)=c$.
Refer to \cite{L10} for the local existence of weak solutions in a two-dimensional bounded domain. Numerical investigations in a two-dimensional bounded domain can be found in \cite{CFKLM12}.
In the whole space $\mathbb{R}^2$, the global well-posedness of smooth solutions was obtained in \cite{CKL13} and \cite{CKL14} under different assumptions on the size of initial data, $\chi$, and $k$. \cite{CKL13} deals with any large initial data but requires $\chi$ and $k$ to satisfy
\begin{equation*}
    \chi,\,k,\,\chi',\,k'\ge0,\quad \sup|\chi-\mu k|<\epsilon\;\, \text{for some} \;\mu>0, 
\end{equation*}
while \cite{CKL14} imposes $\normif{c(t=0)}\ll 1$ but only assumes that
\begin{equation*}
    \chi,\,k,\,\chi',\,k'\ge0.
\end{equation*}
For asymptotic behaviors of smooth solutions in $\mathbb{R}^2$, refer to \cite{CKL14, CKL16, CKL18}.  See also \cite{DLM10} for the global existence of weak solutions in $\mathbb{R}^2$. 
There have been a small number of results of the three-dimensional case of \eqref{KS-N} compared with the two-dimensional one. The only available result of long-time dynamics of smooth solutions for large initial data is the eventual smoothness and stabilization of solutions in a three-dimensional bounded domain (\cite{W17}). See also \cite{TW12} for the same result of subsystem of \eqref{KS-N} obtained by taking $u\equiv 0$, $\chi\equiv 1$, and $k(c)=c$:
\begin{equation}\label{KS}\tag{KS}
    \left\{
    \begin{aligned}
    &\partial_{t} \rho = \nu_{\rho}\Delta \rho-\nabla \cdot(\rho \nabla c),\\
    &\partial_{t} c = \nu_{c}\Delta c - c \rho
    \end{aligned}
    \right.
\end{equation}
in a three-dimensional bounded domain.
For small initial data, global well-posedness of smooth solutions in $\mathbb{R}^3$ was established in \cite{DLM10} and \cite{BK22}. Furthermore, the authors of \cite{BK22} showed the existence of unique global self-similar solutions for small data in scaling invariant Besov spaces. The global existence of weak solutions in a three-dimensional bounded domain and $\mathbb{R}^3$ can be found in \cite{W16} and \cite{KLW22}, respectively.

\medskip

\noindent \textbf{Partially or fully inviscid case.}
To the best of the author's knowledge, there have been three results of long-time dynamics of partially or fully inviscid variations of \eqref{KS-N}. The paper \cite{CKL14} established the global existence of smooth solutions in $\mathbb{R}^2$ of the system obtained from \eqref{KS-N} on neglecting $\Delta c$, provided that $\normif{\rho(t=0)}\ll 1$. Recently, the authors of \cite{IJ21} showed the existence of $C^{\infty}$-data satisfying some vanishing conditions such that the corresponding solutions of the fully inviscid case of \eqref{KS} in $d$-dimensional domain with $d\ge1$ :
\begin{equation*}
    \left\{
    \begin{aligned}
    &\partial_{t} \rho = - \nabla \cdot \left (\rho \nabla c\right),\\
    &\partial_{t} c = - c \rho
    \end{aligned}
    \right.
\end{equation*}
becomes singular in finite time. 
The paper \cite{JKNA23} considered logarithmic sensitivity $\chi(c)=c^{-1}$ and showed finite-time singularity formation of the following system:
\begin{equation*}
    \left\{
    \begin{aligned}
    &\partial_{t} \rho = - \nabla \cdot \left (c^{-1}\rho \nabla c\right),\\
    &\partial_{t} c = - c \rho,
    \end{aligned}
    \right.
\end{equation*}
for nonvanishing initial data. The name `logarithmic sensitivity' originated from $c^{-1}\nabla c=\nabla \log c$, and it is closely related to tumor angiogenesis, an example of chemotaxis. (See \cite{LSN00, WXY16, JKNA23} and references therein.)

\medskip

\noindent \textbf{Partially inviscid 2D Boussinesq system.}
The 2D Boussinesq system for incompressible fluid takes the form
\begin{equation*}
    \left\{
    \begin{aligned}
    &\partial_{t} \rho  + u \cdot \nabla \rho= \nu_{\rho}\Delta \rho,\\
    &\partial_{t} u + u \cdot \nabla u + \nabla p = \nu_{u}\Delta u +\rho e_2,\\
    &\nabla \cdot u=0
    \end{aligned}
    \right.
\end{equation*}
where $e_2=(0,1)$. This system plays a crucial role in the atmospheric sciences and is closely related to 3D Euler and Navier-Stokes equations. The global regularity of the system is well-known when $\nu_{\rho},\,\nu_{u}>0$ in \cite{Cannon80}. Even in partially inviscid cases, i.e., $\nu_{\rho}=0,\,\nu_{u}>0$ or $\nu_{\rho}>0,\,\nu_{u}=0$, the system enjoys global well-posedness (\cite{C06}).

\subsection{Main results and Discussion}
The goal of this paper is to establish the global well-posedness of smooth solutions for \eqref{KS-E} in $\mathbb{R}^2$.
Our physical motivation for replacing the Navier-Stokes equations with the Euler equations is that a fluid having very low viscosity was used in the experiment conducted in \cite{Tuval05}. From a mathematical point of view, the global regularity result of the partially inviscid 2D Boussinesq system arouses curiosity about that of \eqref{KS-E}. Moreover, norm growths of solutions to \eqref{KS-E} could lead to instabilities in \eqref{KS-N} with very small $\nu_u$, which was investigated numerically in \cite{Tuval05}. This kind of result can be observed in other hyperbolic-elliptic type Keller-Segel equations (see \cite{Winkler14, KKS16}).

In this note, we shall always assume that $\rho, c\ge0$, which is preserved by the dynamics of \eqref{KS-E}. Moreover, since our main purpose is not to relax conditions imposed on $\chi$ and $k$, but to show that the diffusion effect in the fluid equation is not a key factor determining global well-posedness of \eqref{KS-N} in $\mathbb{R}^2$, we shall assume throughout that
\begin{equation}\label{intro}
    \chi(c)\equiv 1, \quad k(c)=c,
\end{equation}
which can be considered prototypical. Upon these assumptions, we firstly prove the local existence of smooth solutions for arbitrary smooth initial data in Section \ref{lwp-proof}.
\begin{theorem}[Local well-posedness]\label{thm: lwp}
Let $m>3$. Assume that $\chi$ and $k$ are defined by \eqref{intro}, and initial data $\left (\rho_{0}, c_{0}, u_{0}\right)$ and $\phi$ satisfy
\begin{equation}\label{assumption thm 1.1}
    \left\{
    \begin{aligned}
    &\rho_{0}\ge 0, \, c_{0}\ge 0,\\
    &\left (\rho_{0}, c_{0}, u_{0}\right) \in \left (H^m \times H^{m+1} \times H^{m+1}\right) (\mathbb{R}^2),\\
    &\norm{\nabla^{\alpha} \phi}_{L^{\infty}\left (\mathbb{R}^2\right)} < \infty \quad \text{for} \;\, 1\le |\alpha| \le m+2.
    \end{aligned}
    \right.
\end{equation}
Then there exist some $T_{loc}>0$ and a unique solution $\left (\rho,c,u\right)$ with initial data $\left (\rho_{0}, c_{0}, u_{0}\right)$ for \eqref{KS-E} satisfying
\begin{equation}\label{result thm 1.1}
    \left\{
    \begin{aligned}
    &\left (\rho, c, u \right) \in C\left ([0,T_{loc}]; \left (H^m \times H^{m+1} \times H^{m+1}\right )(\mathbb{R}^2)\right ),\\
    &\left (\nabla \rho, \nabla c\right) \in L^{2}\left ([0,T_{loc}]; \left (H^m \times H^{m+1}\right  )(\mathbb{R}^2\right)).
    \end{aligned}
    \right.
\end{equation}
\end{theorem}

\medskip

Next, in Section \ref{gwp-proof}, we show that these smooth solutions can be extended globally, provided that $\normifr{c_0}$ is sufficiently small.
\begin{theorem}[Global well-posedness]\label{thm: gwp}
Let the assumptions in Theorem \ref{thm: lwp} holds. Then there exists a constant $\delta_0>0$ such that if $\normifr{c_0}\le \delta_0$, then the unique solution in Theorem \ref{thm: lwp} exists globally in time.
\end{theorem}
Compared with the aforementioned global regularity results of \eqref{KS-N} in $\mathbb{R}^2$  (\cite{CKL13, CKL14}), it is natural to assume $\normif{c_0}\ll 1$ since we adopted $\chi, \, k$ as \eqref{intro}. Moreover, recalling the global well-posedness results of 2D partially inviscid Boussinesq system, it is surprising that \eqref{KS-E} also enjoys global regularity even though the equation of $\rho$ involves many derivatives and $\rho,\,c$ complicatedly interact to each other.
Moreover, we can compare Theorem \ref{thm: gwp} with the aformentioned global regularity results in \cite{CKL14} where the authors dealt with the case of $\nu_{c}=0$. In that case, although $\Delta c$ dropped, the remaining term $-c\rho$ has a dissipation effect on solutions, which enables us to expect global well-posedness of the system. However, deleting the diffusion effect of $u$ causes a significant loss of regularity, which makes our result more interesting.

Main ingredient of the proof of Theorem \ref{thm: gwp} is to use the Sobolev embedding $W^{1,q}(\mathbb{R}^2)\hookrightarrow L^{\infty}(\mathbb{R}^2)$ with $q>2$ and to apply the sharp version of Sobolev inequality, the Br\'ezis-Wainger inequality (see \cite{BW80} or \cite{E89}):
\begin{equation}\label{ine: B-W}
    \normifr{f} \lesssim \normbr{f} + \left(1+\normbr{\nabla f}\right)\left(1 + \log_+ \normqr{\nabla f} \right)^{\frac{1}{2}}  \;\, (q>2),
\end{equation}
which are motivated by \cite{C06}. Employing this sharp inequality, we can overcome the loss of regularity which results from the absence of $\Delta u$.

\subsection*{Notation}
We employ the letter $C=C(a,b,\cdots)$ to denote any constant depending on $a,b,\cdots$, which may change from line to line in a given computation. We sometimes use $A\lesssim B$, which means $A\le CB$ for some constant $C$.

\section{Local well-posedness}\label{lwp-proof}
In this section, we prove Theorem \ref{thm: lwp}. Hereafter, without loss of generality, we assume that $\nu_\rho=\nu_c=1$.
We shall divide the proof into two steps, which correspond to the existence and uniqueness of a solution, respectively.\\

\noindent $\mathbf{Step\,1.\;\, Existence}$

\noindent The proof of existence can be done using viscous approximation.
Given fixed $m>3$, $\epsilon>0$ and initial data $(\rho_0,c_0,u_0) \in H^{m} \times H^{m+1} \times H^{m+1}$, we first mollify it in a way that 
\begin{equation}\label{good initial}
    (\rho_{0,\epsilon},c_{0,\epsilon},u_{0,\epsilon}) \in H^{\infty} \times H^{\infty} \times H^{\infty}
\end{equation}
with $H^\infty := \bigcap_{k=0}^{\infty}H^k$ and $(\rho_{0,\epsilon},c_{0,\epsilon},u_{0,\epsilon})$ converges to $(\rho_0,c_0,u_0)$ strongly in the norm $H^{m} \times H^{m+1} \times H^{m+1}$ as $\epsilon \rightarrow 0$. Furthermore, the strong convergence of $(\rho_{0,\epsilon},c_{0,\epsilon},u_{0,\epsilon})$ to $(\rho_0,c_0,u_0)$ in $H^{m} \times H^{m+1} \times H^{m+1}$ enables us to take a sufficiently small $\epsilon_0>0$ such that
\begin{equation}\label{bound of good initial}     
\sup_{\epsilon\in(0,\epsilon_0]}\left(\normh{\rho_{0,\epsilon}}^2 + \normhh{c_{0,\epsilon}}^2 + \normhh{u_{0,\epsilon}}^2 \right) \le 2\left(\normh{\rho_0}^2 + \normhh{c_0}^2 + \normhh{u_0}^2 \right).
\end{equation}
This mollification is done simply to make initial data belong to $H^k$ for any large $k$.

Now we consider
\begin{equation}\label{KS-Ns-epsilon}
    \left\{
    \begin{aligned}
    &\partial_{t} \rho_{\epsilon} + u_{\epsilon} \cdot \nabla \rho_{\epsilon} = \Delta \rho_{\epsilon} - \nabla \cdot \left (\rho_{\epsilon} \nabla c_{\epsilon}\right),\\
    &\partial_{t} c_{\epsilon} + u_{\epsilon} \cdot \nabla c_{\epsilon} = \Delta c_{\epsilon} - c_{\epsilon} \rho_{\epsilon},\\
    &\partial_{t} u_{\epsilon} + u_{\epsilon} \cdot \nabla u_{\epsilon} + \nabla p_{\epsilon} = \epsilon\Delta u_{\epsilon} + \rho_{\epsilon} \nabla \phi,\\
    &\nabla \cdot u_{\epsilon} = 0, \\
    &\rho_{\epsilon}(t=0)=\rho_{0,\epsilon},\;\, c_{\epsilon}(t=0)=c_{0,\epsilon},\;\, u_{\epsilon}(t=0)=u_{0,\epsilon}.\\
    \end{aligned}
    \right.
\end{equation}
The local well-posedness of \eqref{KS-Ns-epsilon} in $\mathbb{R}^2$ was already established in \cite{CKL13}:
\begin{proposition}\label{proposition}
   For any $\epsilon >0$ and any $k>0$, there exist a unique classical solution $(\rho_{\epsilon},c_{\epsilon},u_{\epsilon})$ of \eqref{KS-Ns-epsilon} and its maximal time of existence $T_{max}(\epsilon,k)>0$ such that
\begin{equation*}
    \left (\rho_{\epsilon}, c_{\epsilon}, u_{\epsilon} \right) \in C^1\left ([0,T_{max}(\epsilon,k)); H^k \times H^{k+1} \times H^{k+1}\right ).
\end{equation*}
\end{proposition}
We now fix sufficiently large $k=k_0$ so that the computations in the following estimates of $(\rho_\epsilon,c_\epsilon,u_\epsilon)$ over the time interval $[0,T_{max}(\epsilon,k_0))$ can be justified ($k_0>m+3$ is sufficient by the Sobolev embedding $H^{k_0}(\mathbb{R}^2)\hookrightarrow C^{m+2}(\mathbb{R}^2)$ for $k_0>m+3$). In these estimates, $C$'s represent constants independent of $\epsilon$.
To begin with, we estimate $\normh{\rho_{\epsilon}}$. We recall the Kato-Ponce commutator estimate in \cite{KPV91}:
\begin{equation}\label{comm-1}
    \sum_{{0\le |\alpha| \le m}} \normp{\nabla^{\alpha} \left (fg\right  )- f \nabla^{\alpha} g} \lesssim\norm{\nabla f}_{L^{p_1}}\norm{g}_{W^{m-1,p_2}} + \norm{g}_{L^{p_3}}\norm{f}_{W^{m,p_4}},
\end{equation}
with $p,p_2, p_3 \in \left (1,\infty\right)$ such that $\frac{1}{p} = \frac{1}{p_1} + \frac{1}{p_2} = \frac{1}{p_3} +\frac{1}{p_4}$.
Using $\nabla \cdot u_{\epsilon}=0$ and \eqref{comm-1}, we obtain 
\begin{align*}
    |\left\langle u_{\epsilon} \cdot \nabla \rho_{\epsilon} , \rho_{\epsilon} \right\rangle_{H^m}| 
    &= \left| \sum_{{0\le |\alpha| \le m}} \int \nabla^{\alpha}\left (u_{\epsilon} \cdot \nabla \rho_{\epsilon} \right  ) \nabla^{\alpha}\rho_{\epsilon}\;\right |\\
&=\left| \sum_{{0\le |\alpha| \le m}} \int \left (\nabla^{\alpha}\left (u_{\epsilon} \cdot \nabla \rho_{\epsilon} \right  ) - u_{\epsilon} \cdot \nabla^{\alpha} \nabla \rho_{\epsilon}\right  ) \nabla^{\alpha}\rho_{\epsilon}\;\right | \\
&\le C\left (\normif{\nabla u_{\epsilon}}\normh{\rho_{\epsilon}} + \normd{\nabla \rho_{\epsilon}}\norm{u_{\epsilon}}_{W^{m,4}}\right  )\normh{\rho_{\epsilon}}\\
&\le C\left(\normif{\nabla u_{\epsilon}}\normh{\rho_{\epsilon}}^2 + \normd{\nabla \rho_{\epsilon}}\normh{u_{\epsilon}}^{\frac{1}{2}}\normhh{u_{\epsilon}}^{\frac{1}{2}}\normh{\rho_{\epsilon}}\right)\\
&\le C\left(\left (\normif{\nabla u_{\epsilon}} + \normd{\nabla \rho_{\epsilon}}^2\right  )\normh{\rho_{\epsilon}}^2 +\normhh{u_{\epsilon}}^2\right),
\end{align*}
where we applied the Ladyzhenskaya inequality:
\begin{equation}\label{ine: sob-1}
    \normdr{f} \lesssim \normbr{f}^\frac12 \normbr{\nabla f}^\frac12
\end{equation}
to the fourth line.
On the other hand, we recall an elementary product estimate in \cite{KPV91}:
\begin{equation}\label{pro. est.}
    \normhs{fg} \lesssim \normhs{f}\normif{g} + \normhs{g}\normif{f}\;\,\text{with}\;\,s>0.
\end{equation}
This gives
\begin{align*}
    |\left\langle \rho_{\epsilon} \nabla c_{\epsilon} , \nabla \rho_{\epsilon} \right\rangle_{H^m}| &
    \le C \normh{\rho_{\epsilon}\nabla c_{\epsilon}}\normh{\nabla \rho_{\epsilon}} \\
    &\le C \left( \normif{ \rho_{\epsilon}}\normh{\nabla c_{\epsilon}} + \normh{\rho_{\epsilon}}\normif{\nabla c_{\epsilon}}\right  )\normhh{\rho_{\epsilon}}\\
    &\le C\left (\normif{\rho_{\epsilon}}^2 \normhh{c_{\epsilon}}^2 + \normif{\nabla c_{\epsilon}}^2 \normh{\rho_{\epsilon}}^2\right  ) + \frac{1}{6}\normhh{\rho_{\epsilon}}^2.
\end{align*}
Hence, we have
\begin{equation*}
    \begin{split}
        \frac{1}{2} \frac{d}{dt} \normh{\rho_{\epsilon}}^2 + \normh{\nabla \rho_{\epsilon}}^2 
        &\le |\left\langle u_{\epsilon} \cdot \nabla \rho_{\epsilon} , \rho_{\epsilon} \right\rangle_{H^m}| + |\left\langle \rho_{\epsilon} \nabla c_{\epsilon} , \nabla \rho_{\epsilon} \right\rangle_{H^m}| \\
        &\le C\left ( \left (\normd{\nabla \rho_{\epsilon}}^2 +\normif{\nabla c_{\epsilon}}^2 + \normif{\nabla u_{\epsilon}} \right  )\normh{\rho_{\epsilon}}^2 +\normif{\rho_{\epsilon}}^2\normhh{c_{\epsilon}}^2+\normhh{u_{\epsilon}}^2  \right ) \\
        &\qquad+ \frac{1}{6}\normhh{\rho_{\epsilon}}^2.
    \end{split}
\end{equation*}
With a similar argument, \eqref{comm-1} and \eqref{pro. est.} yield
\begin{equation*}
    \begin{split}
        \frac{1}{2} \frac{d}{dt} \normhh{c_{\epsilon}}^2 + \normhh{\nabla c_{\epsilon}}^2 
        &\le |\left\langle u_{\epsilon} \cdot \nabla c_{\epsilon} , c_{\epsilon} \right\rangle_{H^{m+1}}| + |\left\langle c_{\epsilon} \rho_{\epsilon} , c_{\epsilon} \right\rangle_{H^{m+1}}| \\
        &\le C\left(\left (\normif{\nabla u_{\epsilon}}\normhh{c_{\epsilon}}+\normif{\nabla c_{\epsilon}}\normhh{u_{\epsilon}}\right  )\normhh{c_{\epsilon}}\right. \\
        &\left.\qquad \quad +\left (\normif{c_{\epsilon}}\normhh{\rho_{\epsilon}}+\normif{\rho_{\epsilon}}\normhh{c_{\epsilon}}\right  )\normhh{c_{\epsilon}}\right) \\
        &\le C\left (\left (\normif{\rho_{\epsilon}}+\normif{c_{\epsilon}}^2+\normif{\nabla c_{\epsilon}}^2 + \normif{\nabla u_{\epsilon}}\right  )\normhh{c_{\epsilon}}^2 + \normhh{u_{\epsilon}}^2\right  ) \\
        &\qquad+ \frac{1}{6}\normhh{\rho_{\epsilon}}^2,
    \end{split}
\end{equation*}
and 
\begin{equation*}
    \begin{split}
        \frac{1}{2} \frac{d}{dt} \normhh{u_{\epsilon}}^2 + \epsilon\normhh{\nabla u_{\epsilon}}^2
        &\le |\left\langle u_{\epsilon} \cdot \nabla u_{\epsilon} , u_{\epsilon} \right\rangle_{H^{m+1}}| + |\left\langle \rho_{\epsilon} \nabla \phi , u_{\epsilon} \right\rangle_{H^{m+1}}| \\
        & \le C\left(\normif{\nabla u_{\epsilon}}\normhh{u_{\epsilon}}^2 + \normhh{\rho_{\epsilon}}\normhh{u_{\epsilon}} \right) \\
        &\le C\left (1+\normif{\nabla u_{\epsilon}}\right)\normhh{u_{\epsilon}}^2 + \frac{1}{6}\normhh{\rho_{\epsilon}}^2.
    \end{split}
\end{equation*}
Combining all, we obtain that for $t\in [0,T_{max}(\epsilon,k_0))$,
\begin{equation}\label{est: local}
    \begin{split}
        \frac{d}{dt}X_{\epsilon}+ \normh{\nabla \rho_{\epsilon}}^2 + &\normhh{\nabla c_{\epsilon}}^2 \\
        &\le C \left (1+\normif{\rho_{\epsilon}}+\normif{\rho_{\epsilon}}^2+\normd{\nabla \rho_{\epsilon}}^2+\normif{c_{\epsilon}}^2+\normif{\nabla c_{\epsilon}}^2+\normif{\nabla u_{\epsilon}}\right) X_{\epsilon},
    \end{split}
\end{equation}
where $X_{\epsilon}(t)$ is defined by
\begin{equation*}\label{definition of X}
X_{\epsilon}(t):=\left(\normh{\rho_{\epsilon}}^2+\normhh{c_{\epsilon}}^2+\normhh{u_{\epsilon}}^2\right)(t)
\end{equation*}
on $[0,T_{max}(\epsilon,k_0))$.
Since $m>3$, the Sobolev embeddings, $H^1\left (\mathbb{R}^2\right) \hookrightarrow L^{q}\left (\mathbb{R}^2\right)$ with $q\in[2,\infty)$ and $H^2\left (\mathbb{R}^2\right) \hookrightarrow L^{\infty}\left (\mathbb{R}^2\right)$, provide us with
\begin{equation}\label{local final X bound}
    \frac{d}{dt}X_{\epsilon} + \normh{\nabla \rho_{\epsilon}}^2 + \normhh{\nabla c_{\epsilon}}^2 \le C X_{\epsilon}^2
\end{equation}
on $[0,T_{max}(\epsilon,k_0))$.
We now claim that for any $\epsilon\in(0,\epsilon_0]$,
\begin{equation}\label{definition of Tloc}
    T_{max}(\epsilon,k_0)> T_{loc}:=\frac{1}{4C\left(\normh{\rho_0}^2+\normhh{c_0}^2+\normhh{u_0}^2 \right)},
\end{equation}
where $\epsilon_0$ is from \eqref{bound of good initial}. Suppose, toward the contrary, that there exists $\epsilon_1\in (0,\epsilon_0]$ such that $T_{max}(\epsilon_1,k_0)\le T_{loc}$. Then we have
$T_{max}(\epsilon_1,k_0)\le T_{loc} \le \frac{1}{2CX_{\epsilon_1}(0)}$ by \eqref{bound of good initial}, and therefore \eqref{local final X bound}, the ODE comparison principle, and \eqref{bound of good initial} give 
\begin{equation*}
    X_{\epsilon_1}(t) + \int_{0}^{t} \normh{\nabla \rho_{\epsilon_1}(s)}^2 + \normhh{\nabla c_{\epsilon_1}(s)}^2 \; ds \le 2X_{\epsilon_1}(0) \le 4\left(\normh{\rho_0}^2+\normhh{c_0}^2+\normhh{u_0}^2 \right)
\end{equation*}
on $[0,T_{max}(\epsilon_1,k_0))$. But this is a contradiction to the definition of $T_{max}(\epsilon_1,k_0)$, so that we have shown \eqref{definition of Tloc}. The upshot is that we have $T_{loc}$ independent of $\epsilon$, and the uniform-in-$\epsilon$ bound
\begin{equation}\label{bound of X}
    X_{\epsilon}(t) + \int_{0}^{t} \normh{\nabla \rho_{\epsilon}(s)}^2 + \normhh{\nabla c_{\epsilon}(s)}^2 \; ds \le 4\left(\normh{\rho_0}^2+\normhh{c_0}^2+\normhh{u_0}^2 \right)
\end{equation}
on $[0,T_{loc}]$ whenever $\epsilon\in(0,\epsilon_0]$.
In this way, we conclude that for any $\epsilon\in(0,\epsilon_0]$, $(\rho_{\epsilon},c_{\epsilon},u_{\epsilon})$ and $(\nabla\rho_{\epsilon},\nabla c_{\epsilon})$ are uniformly bounded in 
\begin{equation*}
    C \left ([0,T_{loc}]; H^{m} \times H^{m+1} \times H^{m+1}\right) \quad \text{and} \quad L^2 \left ([0,T_{loc}]; H^{m} \times H^{m+1} \right),
\end{equation*}
repectively. Moreover, from the equation \eqref{KS-Ns-epsilon}, one can see that $(\partial_t\rho_{\epsilon},\partial_tc_{\epsilon},\partial_tu_{\epsilon})$ is uniformly bounded in $C \left ([0,T_{loc}]; L^2 \times L^2 \times L^2\right)$.
Hence we use the Banach-Alaoglu theorem, Aubin-Lions compactness lemma, and diagonal argument to obtain a limit point $(\rho,c,u)$ 
and a subsequence $\epsilon_n \rightarrow 0$ as $n\rightarrow \infty$ such that
\begin{equation}\label{many convergences}
    \left\{\begin{aligned}
         (\rho_{\epsilon_n},c_{\epsilon_n},u_{\epsilon_n}) \rightharpoonup^* (\rho,c,u) \quad \text{weakly-* in } & C \left ([0,T_{loc}]; H^{m} \times H^{m+1} \times H^{m+1}\right) \\
         (\rho_{\epsilon_n},c_{\epsilon_n},u_{\epsilon_n}) \rightarrow (\rho,c,u) \quad \text{strongly in } & C \left ([0,T_{loc}]; H^{s}_{loc}\times H^{s+1}_{loc} \times H^{s+1}_{loc}\right) \quad (\forall s<m) \\
         (\nabla\rho_{\epsilon_n},\nabla c_{\epsilon_n}) \rightharpoonup (\nabla\rho,\nabla c) \quad \text{weakly in } &L^2 \left ([0,T_{loc}]; H^{m} \times H^{m+1} \right)
    \end{aligned}
    \right.
\end{equation}
In particular, since $m> 3$, the Sobolev embedding $H^s\left (\mathbb{R}^2\right) \hookrightarrow C^2\left (\mathbb{R}^2\right) $ $(s>3)$ guarantees that the above strong convergence is in $C \left ([0,T_{loc}]; C^2_{loc}\times C^3_{loc} \times C^3_{loc}\right)$, which is enough to ensure pointwise convergence of each term in \eqref{KS-E} with \eqref{intro}. Thus using convergences in \eqref{many convergences}, we can show that the limit $(\rho,c,u)$ is a solution to \eqref{KS-E} with \eqref{intro} and satisfies
\begin{equation*}
    (\rho,c,u) \in C \left ([0,T_{loc}]; H^{m} \times H^{m+1} \times H^{m+1}\right), \quad  (\nabla \rho ,\nabla c) \in L^2 \left ([0,T_{loc}]; H^{m} \times H^{m+1} \right).
\end{equation*}

\medskip

\noindent $\mathbf{Step\,2.\;\, Uniqueness}$

\noindent The uniqueness of the above local-in-time solution is derived as follows. Assume that there exist two local-in-time solutions $(\rho_1,c_1,u_1)$ and $(\rho_2,c_2,u_2)$ of \eqref{KS-E} with \eqref{intro} sharing the same initial data such that properties stated in Theorem \ref{thm: lwp} are satisfied. Let 
\begin{equation*}
    (\Tilde{\rho},\Tilde{c},\Tilde{u},\tilde{p}):= (\rho_1-\rho_2,c_1-c_2,u_1-u_2,p_1-p_2).
\end{equation*}
Then $(\Tilde{\rho},\Tilde{c},\Tilde{u})$ solves
\begin{equation*}
    \left\{
    \begin{aligned}
    &\partial_{t} \Tilde{\rho} + u_1 \cdot \nabla \Tilde{\rho} + \Tilde{u}\cdot \nabla \rho_2 = \Delta \Tilde{\rho} - \nabla \cdot \left (\rho_1 \nabla \Tilde{c}\right)  + \nabla \cdot \left (\Tilde{\rho} \nabla c_2\right),\\
    &\partial_{t} \Tilde{c} + u_1 \cdot \nabla \Tilde{c} + \Tilde{u}\cdot \nabla c_2 = \Delta \Tilde{c} -c_1 \Tilde{\rho} - \Tilde{c} \rho_2,\\
    &\partial_{t} \Tilde{u} + u_1 \cdot \nabla \Tilde{u} + \Tilde{u}\cdot \nabla u_2 + \nabla \tilde{p} = \Tilde{\rho} \nabla \phi,\\
    &\nabla \cdot \tilde{u} = 0, \\
    &\Tilde{\rho}(t=0)=0,\;\, \Tilde{c}(t=0)=0,\;\, \Tilde{u}(t=0)=0.\\
    \end{aligned}
    \right.
\end{equation*}
Employing $u_1$ and $\tilde{u}$ is divergence-free, we estimate from the above equation as follows:
\begin{equation*}
\begin{split}
    \frac12 \frac{d}{dt}\normb{\tilde{\rho}}^2 + \normb{\nabla\tilde{\rho}}^2 &\le \normb{\Tilde{u}}\normif{\nabla\rho_2}\normb{\tilde{\rho}} + \normif{\rho_1}\normb{\nabla \tilde{c}}\normb{\nabla \tilde {\rho}}
    +\normb{\tilde{\rho}}\normif{\nabla c_2}\normb{\nabla \tilde {\rho}} \\
    &\le C\left(1+\normif{\nabla\rho_2}^2+\normif{\rho_1}^2+\normif{\nabla c_2}^2\right)\left(\normb{\tilde{\rho}}^2 + \normb{\Tilde{u}}^2+ \normb{\nabla \tilde{c}}^2  \right) + \frac16 \normb{\nabla \tilde {\rho}}^2,
\end{split}
\end{equation*}
\begin{equation*}
\begin{split}
    \frac12 \frac{d}{dt}\normb{\tilde{c}}^2 + \normb{\nabla\tilde{c}}^2 &\le \normb{\Tilde{u}}\normif{\nabla c_2}\normb{\tilde{c}} + \normif{c_1}\normb{\tilde{\rho}}\normb{\tilde{c}} 
    +\normif{\rho_2}\normb{\tilde{c}}^2 \\
    &\le \left(1+\normif{\nabla c_2}^2+\normif{c_1}^2+\normif{\rho_2}\right)\left(\normb{\Tilde{u}}^2+ \normb{\tilde{c}}^2 +\normb{\Tilde{\rho}}^2 \right),
\end{split}
\end{equation*}
\begin{equation*}
\begin{split}
    \frac12 \frac{d}{dt}\normb{\nabla\tilde{c}}^2 + \normb{\nabla^2\tilde{c}}^2 &\le \normif{\nabla u_1}\normb{\nabla \tilde{c}}^2 
    +\normb{\nabla \tilde{u}}\normif{\nabla c_2}\normb{\nabla \tilde{c}}
    +\normb{\tilde{u}}\normif{\nabla^2 c_2}\normb{\nabla \tilde{c}} \\
    &\quad+ \normif{\nabla c_1}\normb{\tilde{\rho}}\normb{\nabla\tilde{c}}
    + \normif{c_1}\normb{\nabla\tilde{\rho}}\normb{\nabla\tilde{c}}
    +\normif{\rho_2}\normb{\nabla \tilde{c}}^2\\
    &\quad+\normb{\tilde{c}}\normif{\nabla\rho_2}\normb{\nabla\tilde{c}} \\
    &\le C\left(1+\normif{\nabla u_1}+\normif{\nabla c_2}^2+\normif{\nabla^2 c_2}^2+\normif{c_1}^2+\normif{\nabla c_1}^2+\normif{\rho_2}+\normif{\nabla\rho_2}^2\right) \\
    &\quad \times\left(\normb{\tilde{\rho}}^2  + \normb{\Tilde{u}}^2+\normb{\nabla\Tilde{u}}^2+ \normb{\tilde{c}}^2 +\normb{\nabla\tilde{c}}^2 \right) + \frac16\normb{\nabla \tilde{\rho}}^2,
\end{split}
\end{equation*}
and
\begin{equation*}
\begin{split}
    \frac12 \frac{d}{dt}\normb{\tilde{u}}^2 &\le \normif{\nabla u_2}\normb{\Tilde{u}}^2 + \normif{\nabla \phi}\normb{\tilde{\rho}}\normb{\tilde{u}} \\
    &\le C\left(1+\normif{\nabla u_2}\right)\left(\normb{\tilde{\rho}}^2+\normb{\Tilde{u}}^2\right).
\end{split}
\end{equation*}
Furthermore, let us consider the vorticity equations of fluids:
\begin{equation*}\label{eq: vorticity}
    \partial_{t} \omega_i + u_i\cdot \nabla \omega_i = \nabla^{\bot} \rho_i \cdot \nabla \phi,
\end{equation*}
where $\tilde{\omega}_i:=\nabla^{\perp}\cdot \tilde{u}_i$ with $i\in \left\{1,2\right\}$.
Denoting $\tilde{\omega}:=\omega_1-\omega_2$, we obtain
\begin{equation*}
    \partial_{t} \tilde{\omega} + u_1\cdot \nabla \tilde{\omega} + \tilde{u}\cdot \nabla \omega_2 = \nabla^{\bot} \tilde{\rho} \cdot \nabla \phi,
\end{equation*}
which gives
\begin{equation*}
\begin{split}
    \frac12 \frac{d}{dt}\normb{\tilde{\omega}}^2 &\le \normif{\nabla \omega_2}\normb{\Tilde{\omega}}^2 + \normif{\nabla \phi}\normb{\nabla\tilde{\rho}}\normb{\tilde{\omega}} \\
    &\le C\left(1+\normif{\nabla \omega_2}\right)\normb{\Tilde{\omega}}^2 + \frac16 \normb{\nabla\tilde{\rho}}^2.
\end{split}
\end{equation*}
Combining all and using  the Calder\'on-Zygmund inequality: $\normb{\nabla \tilde{u}} \lesssim \normb{\tilde{\omega}}$, we arrive at
\begin{equation*}
    \frac{d}{dt}\tilde{X}(t) \le C F(t) \tilde{X}(t),
\end{equation*}
where
\begin{equation*}
    \tilde{X}(t)= \left(\normb{\tilde{\rho}}^2 + \normb{\tilde{c}}^2 + \normb{\nabla\tilde{c}}^2 +\normb{\tilde{u}}^2+\normb{\tilde{\omega}}^2\right)(t)
\end{equation*}
and
\begin{equation*}
\begin{split}
    F(t)=1+&\left(\normif{\nabla u_1}+\normif{\nabla u_2}+\normif{\nabla \omega_2}+\normif{\nabla c_2}^2+\normif{\nabla^2 c_2}^2 \right.\\
    &\quad\left.+\normif{c_1}^2 
    +\normif{\nabla c_1}^2+\normif{\rho_1}^2+\normif{\rho_2}+\normif{\nabla\rho_2}^2\right)(t).
\end{split}
\end{equation*}
Note that the Sobolev embedding enables us to bound $\sup_{0\le t\le T_{loc}}F(t)$, where the time $T_{loc}$ is from \eqref{result thm 1.1}. Since $X(0)=0$, the Gr\"onwall's inequality gives us $X(t)=0$ on $[0,T_{loc}]$.
$\Box$

\section{Global well-posedness}\label{gwp-proof}
In this section, we show Theorem \ref{thm: gwp}. 
Our strategy involves employing the viscous approximation as in Step 1 of the last section. But in this time we rely on the following global well-posedness result of \eqref{KS-Ns-epsilon} in $\mathbb{R}^2$ from \cite{CKL14}:
\begin{proposition}\label{proposition-2}
    For any $\epsilon>0$ and any $k>0$, there exists a constant $\delta>0$ independent of $\epsilon$  such that if $\normif{c_{0,\epsilon}}\le \delta$, then there exists a unique classical solution $(\rho_{\epsilon},c_{\epsilon},u_{\epsilon})$ of \eqref{KS-Ns-epsilon} such that
\begin{equation}\label{regualrity for global}
    \left (\rho_{\epsilon}, c_{\epsilon}, u_{\epsilon} \right) \in C^1\left ([0,\infty); H^k \times H^{k+1} \times H^{k+1}\right ).
\end{equation}
\end{proposition}
\begin{remark}
    Here it is important that we can choose $\delta$ independent of $\epsilon$. Since this part is not explicitly revealed in \cite{CKL14}, we shall provide an explanation here.  
    In \cite{CKL14}, the authors essentially derived the following blow-up criterion for \eqref{KS-Ns-epsilon} with $\epsilon>0$ (Theorem 1 in \cite{CKL14}):

    If $T_{max}(\epsilon,k)$ is a maximal time of existence with $T_{max}(\epsilon,k)<\infty$, then
    \begin{equation*}
        \norm{\rho_\epsilon}_{L^q(0,T_{max}(\epsilon,k);L^p)}=\infty, \quad \frac1p+\frac1q \le 1,\quad 1<p\le \infty.
    \end{equation*}
    Then the authors chose $\delta=\frac{1}{24p}$ with $p\in(1,\infty)$ and showed (Proposition 2 in \cite{CKL14}):
    
    If $\normif{c_{0,\epsilon}}\le \delta$, then
    \begin{equation}\label{inequality in the first remark}
        \norm{\rho_\epsilon(t)}_{L^p}\le e^{6(p-1)\normif{c_{0,\epsilon}}}\norm{\rho_{0,\epsilon}}_{L^p}
    \end{equation}
    for $t\in[0,T_{max}(\epsilon,k))$, which implies $T_{max}(\epsilon,k)=\infty$ by the above blow-up criterion. The key point here is that in order to obtain \eqref{inequality in the first remark}, the authors only made use of the incompressibility of $u_\epsilon$, the equation of $\rho_{\epsilon}$, and $\normif{c_{0,\epsilon}}\le \delta$, so that $\epsilon>0$ from the term $\epsilon \Delta u_\epsilon$ in \eqref{KS-Ns-epsilon} does not have any impact on deriving \eqref{inequality in the first remark}. One can refer to \cite{CKL14} for details.
\end{remark}
We now fix sufficiently large $k$ and assume $\normif{c_{0,\epsilon}}\le \delta$ in Proposition \ref{proposition-2} to obtain the solution $(\rho_\epsilon,c_\epsilon,u_\epsilon)$ of \eqref{KS-Ns-epsilon} satisfying the regualrity \eqref{regualrity for global}. This regularity allows us to justify the computations in the following estimates of $(\rho_\epsilon,c_\epsilon,u_\epsilon)$. Throughout these estimates, $C$'s represent constants independent of $\epsilon$. 
Moreover, we only consider sufficiently small $\epsilon>0$ so that we can use \eqref{bound of X}.

Now let $T^*\in (0,\infty)$ be an arbitrary time. Since we already have the estimate \eqref{bound of X} on $[0,T_{loc}]$, we shall consider arbitrary $T^*\in (T_{loc},\infty)$. Our aim is to show the estimate \eqref{final estimate} on $[0,T^*]$ with $C>0$ uniformly bounded for $\epsilon>0$.

To begin with, we observe the following properties:
\begin{equation}\label{est: c}
     \sup_{t\in[0,T^*]}\normq{ c_{\epsilon}(t)}  \le  
 \normq{c_{0,\epsilon}}\le C\normq{c_0} \quad \text{for any}\;\,q \in [1,\infty],
\end{equation}
and
\begin{equation}\label{est: c-1-2-2}
     \norm{\nabla  c_{\epsilon}}_{L^2([0,T^*];L^2)} \le  \normb{c_{0,\epsilon}}\le C\normb{c_0}. 
\end{equation} 
Indeed, taking $L^2$ inner product of the equation of $ c_{\epsilon}$ with $ c_{\epsilon}^{q-1}$, the assumptions that $\nabla \cdot u_{\epsilon}=0$ and $\rho_\epsilon, c_\epsilon \ge0$ give us
\begin{equation*}
    \begin{split}
        \frac{1}{q} \frac{d}{dt} \normq{ c_{\epsilon}}^q + \frac{4(q-1)}{q^2} \normb{\nabla{ c_{\epsilon}^{\frac{q}{2}}}}^2 = -\int  c_{\epsilon}^{q-1}\nabla  c_{\epsilon}\cdot  u_{\epsilon} -\int  c_{\epsilon}^q \rho_{\epsilon}  \le 0,
    \end{split}
\end{equation*}
which implies both \eqref{est: c} and \eqref{est: c-1-2-2}. 

Next, we estimate $\normif{ \rho_{\epsilon}}$.
\begin{lemma}\label{lem: rho-infty}
There exists a constant $\delta_0 >0$ such that if $\normif{c_0}\le \delta_0$, then
\begin{equation}\label{est: rho-infty}
    \sup_{t\in[0,T^*]}\normif{ \rho_{\epsilon}(t)}\le C(\norm{\rho_0}_{L^2\cap L^\infty},\norm{c_0}_{L^2\cap L^\infty},T^*)
\end{equation}
\end{lemma}
\begin{proof}
This follows from the estimate in Theorem 2 of \cite{CKL14}:
\begin{equation}\label{real rho l^infty bound}
    \normif{ \rho_{\epsilon}(t)}\le \frac{C(\norm{\rho_0}_{L^2\cap L^\infty},\norm{c_0}_{L^2\cap L^\infty},T^*)}{(1+t)^{\frac{1}{2}}}\quad \text{for}\;\,t\in[0,T^*],
\end{equation}
which can be proved by modifying the De Giorgi's method introduced in \cite{PV12}. Even though Theorem 2 of \cite{CKL14} deals with the case $\epsilon=1$, its proof does not depend on the coefficient $\epsilon$ of diffusion term $\epsilon\Delta u_\epsilon$ in the equation of fluids. It only made use of the incompressibility of $u_\epsilon$ and equations of $\rho_\epsilon$ and $c_\epsilon$.

For clarity, we give a rough sketchy of proof. Suppose $\normif{c_0}\le \min{\left\{\frac{\delta}{2},\frac{1}{100}\right\}}:=\delta_0$, where $\delta$ is from Proposition \ref{proposition-2}. Then we may assume $\normif{c_{0,\epsilon}}\le \min{\left\{\delta,\frac{1}{50}\right\}}$ by the strong convergence of $(\rho_{0,\epsilon},c_{0,\epsilon},u_{0,\epsilon})$ to $(\rho_0,c_0,u_0)$ in $H^{m} \times H^{m+1} \times H^{m+1}$.
We first show that
\begin{equation}\label{L^2bound of rho}
    \sup_{t\in [0,T^*]}\normb{ \rho_{\epsilon}(t)} \le  Ce^{ C\normif{c_0}^2}\normb{\rho_0}.
\end{equation}
To do so, we define a positive function $\psi( c_{\epsilon})=e^{12 c_{\epsilon}^2}$. Then we can check 
\begin{equation}\label{condition of psi}
    \psi'( c_{\epsilon})\ge 0,\qquad \frac{4(\psi')^2}{\psi} + \psi + \psi' \le \frac{\psi''}{4}.
\end{equation}
On the other hand, we compute
\begin{equation*}
    \begin{split}
        \frac12 \frac{d}{dt} \int  \rho_{\epsilon}^2 \psi( c_{\epsilon}) &=\int  \rho_{\epsilon} \psi( c_{\epsilon}) (- u_{\epsilon}\cdot \nabla  \rho_{\epsilon} +\Delta  \rho_{\epsilon} - \nabla \cdot ( \rho_{\epsilon} \nabla  c_{\epsilon}) )+ \frac12\int  \rho_{\epsilon}^2 \psi'( c_{\epsilon})(- u_{\epsilon}\cdot \nabla  c_{\epsilon} +\Delta  c_{\epsilon} -  \rho_{\epsilon}  c_{\epsilon} ) \\
        &=\int  \rho_{\epsilon} \psi( c_{\epsilon}) \Delta  \rho_{\epsilon} - \int  \rho_{\epsilon} \psi( c_{\epsilon}) \nabla \cdot ( \rho_{\epsilon}\nabla  c_{\epsilon}) + \frac12 \int  \rho_{\epsilon}^2 \psi'( c_{\epsilon}) \Delta  c_{\epsilon}-\frac12 \int  \rho_{\epsilon}^2 \psi'( c_{\epsilon})  \rho_{\epsilon}  c_{\epsilon} \\
        &\quad -\int  \rho_{\epsilon} \psi( c_{\epsilon})  u_{\epsilon} \cdot \nabla  \rho_{\epsilon} -\frac12 \int  \rho_{\epsilon}^2\psi'( c_{\epsilon})  u_{\epsilon}\cdot \nabla  c_{\epsilon}.
    \end{split}
\end{equation*}
We note that $\psi'( c_{\epsilon}) u_{\epsilon}\cdot \nabla  c_{\epsilon} = \nabla \cdot (\psi( c_{\epsilon}) u_{\epsilon})$ due to $\nabla \cdot  u_{\epsilon}=0$, so that last two terms of the above equality are cancelled:
\begin{equation*}
    \int  \rho_{\epsilon}\psi( c_{\epsilon})  u_{\epsilon}\cdot \nabla  \rho_{\epsilon} + \frac12 \int  \rho_{\epsilon}^2 \psi'( c_{\epsilon}) u_{\epsilon}\cdot \nabla  c_{\epsilon}=0.
\end{equation*}
Via the integration by parts, we have
\begin{equation*}
    \begin{split}
        \frac12 \frac{d}{dt} &\int  \rho_{\epsilon}^2 \psi( c_{\epsilon}) + \int \psi( c_{\epsilon}) |\nabla  \rho_{\epsilon}|^2 +\frac12 \int  \rho_{\epsilon}^2 \psi''( c_{\epsilon}) |\nabla  c_{\epsilon}|^2 \\
        &=-2\int  \rho_{\epsilon} \psi'( c_{\epsilon}) \nabla  \rho_{\epsilon} \cdot \nabla  c_{\epsilon}+ \int  \rho_{\epsilon} \psi( c_{\epsilon}) \nabla  \rho_{\epsilon} \cdot \nabla  c_{\epsilon} +\int  \rho_{\epsilon}^2\psi'( c_{\epsilon})|\nabla  c_{\epsilon}|^2-\frac12 \int  \rho_{\epsilon}^2 \psi'( c_{\epsilon})  \rho_{\epsilon}  c_{\epsilon}.
    \end{split}
\end{equation*}
Noticing the last term of the above equality is nonpositive and using Cauchy-Schwartz inequality, we obtain
\begin{equation*}
    \begin{split}
        \frac12 \frac{d}{dt} &\int  \rho_{\epsilon}^2 \psi( c_{\epsilon}) + \frac12\int \psi( c_{\epsilon}) |\nabla  \rho_{\epsilon}|^2 +\frac12 \int  \rho_{\epsilon}^2 \psi''( c_{\epsilon}) |\nabla  c_{\epsilon}|^2 \\
        &\le 4\int  \rho_{\epsilon}^2 \frac{(\psi'( c_{\epsilon}))^2}{\psi( c_{\epsilon})}|\nabla  c_{\epsilon}|^2 + \int  \rho_{\epsilon}^2 \psi( c_{\epsilon}) |\nabla  c_{\epsilon}|^2 +\int  \rho_{\epsilon}^2\psi'( c_{\epsilon})|\nabla  c_{\epsilon}|^2.
    \end{split}
\end{equation*}
Recalling \eqref{condition of psi}, we arrive at
\begin{equation*}
    \frac12 \frac{d}{dt} \int  \rho_{\epsilon}^2 \psi( c_{\epsilon}) + \frac12\int \psi( c_{\epsilon}) |\nabla  \rho_{\epsilon}|^2 +\frac14 \int  \rho_{\epsilon}^2 \psi''( c_{\epsilon}) |\nabla  c_{\epsilon}|^2 \le 0.
\end{equation*}
Since $\psi\ge1$ and 
\begin{equation*}
    \normb{\rho_{0,\epsilon}} \le C \normb{\rho_0},\quad \normif{c_{0,\epsilon}} \le C \normif{c_0},
\end{equation*}
we can obtain \eqref{L^2bound of rho}.

Next we show \eqref{real rho l^infty bound}.
In this time, we shall compute $\frac{d}{dt}\int ( \rho_{\epsilon}-K)^2_+ \psi( c_{\epsilon}) $ for a constant $K>0$. Similarly proceeding as before, we can compute
\begin{equation*}
    \begin{split}
        \frac12 \frac{d}{dt} &\int ( \rho_{\epsilon}-K)^2_+ \psi( c_{\epsilon})
        +\int \psi( c_{\epsilon}) |\nabla( \rho_{\epsilon}-K)_+|^2
        +\frac12 \int ( \rho_{\epsilon}-K)^2_+ \psi''( c_{\epsilon})|\nabla  c_{\epsilon}|^2 \\
        &=-2\int ( \rho_{\epsilon}-K)_+\psi'( c_{\epsilon}) \nabla ( \rho_{\epsilon}-K)_+ \cdot \nabla  c_{\epsilon}
        +\int ( \rho_{\epsilon}-K)_+\psi( c_{\epsilon}) \nabla ( \rho_{\epsilon}-K)_+ \cdot \nabla  c_{\epsilon}  +\int ( \rho_{\epsilon}-K)^2_+ \psi'( c_{\epsilon})|\nabla  c_{\epsilon}|^2\\
        &\quad -\frac12\int ( \rho_{\epsilon}-K)^2_+ \psi'( c_{\epsilon})  \rho_{\epsilon}  c_{\epsilon}
        +K\int \psi( c_{\epsilon})\nabla ( \rho_{\epsilon}- c_{\epsilon})_+ \cdot \nabla  c_{\epsilon}
        +K\int ( \rho_{\epsilon}-K)_+ \psi'( c_{\epsilon})|\nabla  c_{\epsilon}|^2.
    \end{split}
\end{equation*}
We note that the last three integrands in the equality above are bounded as follows:
\begin{equation*}
    \begin{split}
    ( \rho_{\epsilon}-K)^2_+ \psi'( c_{\epsilon})  \rho_{\epsilon}  c_{\epsilon} &\le 0, \\
    \psi( c_{\epsilon})\nabla ( \rho_{\epsilon}- c_{\epsilon})_+ \cdot \nabla  c_{\epsilon} &\le \frac{1}{4K} \psi( c_{\epsilon}) |\nabla( \rho_{\epsilon}-K)_+|^2 + 4K \psi( c_{\epsilon})|\nabla  c_{\epsilon}|^2,\\
        ( \rho_{\epsilon}-K)_+ \psi'( c_{\epsilon})|\nabla  c_{\epsilon}|^2&\le \left(\frac{1}{8K}( \rho_{\epsilon}-K)^2_+ + 8K  \right) \psi'( c_{\epsilon}) |\nabla  c_{\epsilon}|^2.
    \end{split}
\end{equation*}
Using these bounds and \eqref{condition of psi}, we obtain
\begin{equation}\label{degiori1}
    \frac12 \frac{d}{dt} \int ( \rho_{\epsilon}-K)^2_+ \psi( c_{\epsilon})
        +\frac14\int \psi( c_{\epsilon}) |\nabla( \rho_{\epsilon}-K)_+|^2
        +\frac18 \int ( \rho_{\epsilon}-K)^2_+ \psi''( c_{\epsilon})|\nabla  c_{\epsilon}|^2 \le 8K^2L\int |\nabla  c_{\epsilon}|^2,
\end{equation}
where $L:=\sup_{0\le  c_{\epsilon}\le \normif{c_{0,\epsilon}}}(\psi( c_{\epsilon})+\psi'( c_{\epsilon}))$.
Multiplying the equation of $ c_{\epsilon}$ with $16K^2L$, we have
\begin{equation*}
    \frac{d}{dt} \int 8K^2L ( c_{\epsilon}-K)^2_+ 
        + \int 16K^2L |\nabla  c_{\epsilon}|^2 \le 0.
\end{equation*}
Summing this with \eqref{degiori1}, we arrive at
\begin{equation}\label{degiori-2}
\begin{split}
    \frac{d}{dt}  \left(\frac12\int ( \rho_{\epsilon}-K)^2_+ \psi( c_{\epsilon}) + \int 8K^2L ( c_{\epsilon}-K)^2_+ \right) &+\frac14\int \psi( c_{\epsilon}) |\nabla( \rho_{\epsilon}-K)_+|^2 \\
    &+\frac18 \int ( \rho_{\epsilon}-K)^2_+ \psi''( c_{\epsilon})|\nabla  c_{\epsilon}|^2 + 8K^2L\int |\nabla  c_{\epsilon}|^2 \le 0.
\end{split}
\end{equation}
Now we consider
\begin{equation*}
    U_\epsilon(\xi):=\int_0^{T^*}\nu(t)\int ( \rho_{\epsilon}-\xi\eta(t))^2_+ + \int_0^{T^*}\nu(t)\int ( c_{\epsilon}-\xi\eta(t))^2_+ ,\quad \xi>0,
\end{equation*}
where $\eta(t)$ are defined by
\begin{equation*}
    \nu(t)=(1+t)^{-1},\quad \eta(t)=(1+t)^{-\frac12}.
\end{equation*}
Using \eqref{degiori-2} and proceeding in the same manners as the proof of Theorem 2 in \cite{CKL14}, we can show that $U_\epsilon(\xi)$
is finite for every $\xi>0$ as long as $U_\epsilon(\xi)$ exists and 
\begin{equation*}\label{degiori est1}
    U_\epsilon'(\xi) \le -C \xi^{-\frac12}U_\epsilon^{\frac34},\quad \xi>\xi_{0,\epsilon}:=\max \left\{\normif{\rho_{0,\epsilon}}, \normif{c_{0,\epsilon}} \right\}
\end{equation*}
for a constant $C>0$ independent of $\epsilon$. Thus $U_\epsilon(\xi)$ vanishes at a finite value $\xi_{1,\epsilon}$ with 
\begin{equation}\label{degiori est2}
    \xi_{1,\epsilon}\le \left(\frac2C U_\epsilon^\frac14(\xi_{0,\epsilon})+\xi_{0,\epsilon}^\frac12 \right)^2.
\end{equation}
But since $ \rho_{\epsilon},\, c_{\epsilon}\ge0$, \eqref{L^2bound of rho} and \eqref{est: c} give
\begin{equation}\label{degiori est3}
        U_\epsilon(\xi_{0,\epsilon})\le \int_0^{T^*}\int  \rho_{\epsilon}^2 + \int_0^{T^*}\int  c_{\epsilon}^2
        \le  C\left(e^{ C\normif{c_{0}}^2}\normb{\rho_{0}}^2 + \normb{c_0}^2\right)T^*
\end{equation}
 Moreover, we have
 \begin{equation}\label{degiori est 4}
     \xi_{0,\epsilon}=\max \left\{\normif{\rho_{0,\epsilon}}, \normif{c_{0,\epsilon}} \right\} \le C \max \left\{\normif{\rho_0}, \normif{c_0} \right\}.
 \end{equation}
 
Combining \eqref{degiori est2} - \eqref{degiori est 4}, we are done.
\end{proof}

Henceforth, $\delta_0$ represents the constant from Lemma \ref{lem: rho-infty}. We now provide $L^2$-estimate of $( \rho_{\epsilon}, \nabla  c_{\epsilon},  u_{\epsilon},  \omega_{\epsilon})$ and $L^4$-estimate of $ u_{\epsilon}$, where $ \omega_{\epsilon}:=\nabla^{\perp}\cdot  u_{\epsilon}$ denotes the vorticity of fluids.
\begin{lemma}
If $\normif{c_0}\le \delta_0$, then
\begin{subequations}
\begin{align}
    \sup_{t\in[0,T^*]}\left( \normb{ \rho_{\epsilon}(t)}+\normb{\nabla  c_{\epsilon}(t)}+\norm{ u_{\epsilon}(t)}_{H^1}\right) &\le C, \label{est: L^2}\\
     \sup_{t\in[0,T^*]} \normd{ u_{\epsilon}(t)} &\le C, \label{est: u-L^4} \\
     \norm{\nabla  \rho_{\epsilon}}_{L^2([0,T^*];L^2)} + \norm{\nabla^2  c_{\epsilon}}_{L^2([0,T^*];L^2)} &\le C \label{est: nabla^2 c L^2}
\end{align}
\end{subequations}
for a constant $C=C(\norm{\rho_0}_{L^2\cap L^\infty},\norm{c_0}_{L^\infty \cap H^1},\norm{u_0}_{H^1},T^*)$.
\end{lemma}
\begin{proof}
Noticing Lemma \ref{lem: rho-infty}, we have
\begin{equation*}
    \frac{1}{2}\frac{d}{dt}\normb{ \rho_{\epsilon}}^2 + \normb{\nabla  \rho_{\epsilon}}^2 \le \normif{ \rho_{\epsilon}}\normb{\nabla  c_{\epsilon}}\normb{\nabla  \rho_{\epsilon}} \le C\normb{\nabla  c_{\epsilon}}^2 + \frac14 \normb{\nabla  \rho_{\epsilon}}^2.
\end{equation*}
for $C=C\left(\norm{\rho_0}_{L^2\cap L^\infty},\norm{c_0}_{L^2\cap L^\infty},T^*\right).$
Observing the equation of $\nabla  c_{\epsilon}$:
\begin{equation}\label{eq: nb c}
    \partial_{t} \nabla  c_{\epsilon} - \nabla \Delta  c_{\epsilon} = - u_{\epsilon}\cdot \nabla^2  c_{\epsilon} - \nabla  u_{\epsilon} \cdot \nabla  c_{\epsilon}
    -\nabla ( c_{\epsilon} \rho_{\epsilon}),
\end{equation}
we use \eqref{ine: sob-1}, the Calder\'on-Zygmund inequality: $\normb{\nabla  u_{\epsilon}} \lesssim \normb{ \omega_{\epsilon}}$, and \eqref{est: c} to obtain
\begin{align*}
    \frac{1}{2}\frac{d}{dt}\normb{\nabla  c_{\epsilon}}^2 + \normb{\nabla^2  c_{\epsilon}}^2 
    &\le \left | \int \left (\nabla  u_{\epsilon} \cdot \nabla  c_{\epsilon}\right  ) \nabla  c_{\epsilon} \right|
    + \left | \int \nabla ( c_{\epsilon} \rho_{\epsilon}) \nabla  c_{\epsilon} \right| \\
    &\le  \normb{\nabla  u_{\epsilon}}\normd{\nabla  c_{\epsilon}}^2+ \normb{ \rho_{\epsilon}}\normif{ c_{\epsilon}}\normb{\nabla^2  c_{\epsilon}} \\
    &\le C\left(\normb{ \omega_{\epsilon}}\normb{\nabla  c_{\epsilon}}\normb{\nabla^2  c_{\epsilon}} + \normb{ \rho_{\epsilon}}^2\normif{ c_{\epsilon}}^2\right) + \frac{1}{4}\normb{\nabla^2  c_{\epsilon}} \\
    &\le C\left(\normb{ \omega_{\epsilon}}^2\normb{\nabla  c_{\epsilon}}^2 + \normb{ \rho_{\epsilon}}^2\right) + \frac{1}{2}\normb{\nabla^2  c_{\epsilon}}
\end{align*}
for $C=C\left(\normif{c_0}\right).$
The equation of $ u_{\epsilon}$ gives
\begin{equation*}
    \frac{d}{dt} \normb{ u_{\epsilon}}^2 + \epsilon\normb{\nabla u_\epsilon}^2 \le C\left(\normb{ \rho_{\epsilon}}^2 + \normb{ u_{\epsilon}}^2\right).
\end{equation*}
Considering the vorticity equation:
\begin{equation}\label{eq: vorticity}
    \partial_{t}  \omega_{\epsilon} +  u_{\epsilon}\cdot \nabla  \omega_{\epsilon} = \epsilon\Delta \omega_\epsilon +\nabla^{\bot}  \rho_{\epsilon} \cdot \nabla \phi,
\end{equation}
we have
\begin{equation*}
    \frac{1}{2}\frac{d}{dt}\normb{ \omega_{\epsilon}}^2+\epsilon \normb{\nabla \omega_\epsilon}^2 \le C\normb{ \omega_{\epsilon}}^2 + \frac{1}{4} \normb{\nabla  \rho_{\epsilon}}^2.
\end{equation*}
Combining all, we arrive at
\begin{equation*}
\begin{split}
    \frac{d}{dt}\left(\normb{ \rho_{\epsilon}}^2 + \normb{\nabla  c_{\epsilon}}^2+ \normb{ u_{\epsilon}}^2 + \normb{ \omega_{\epsilon}}^2 \right)
    &+\normb{\nabla  \rho_{\epsilon}}^2 + \normb{\nabla^2  c_{\epsilon}}^2 \\
    &\le C\left(1+\normb{\nabla  c_{\epsilon}}^2 \right)\left(\normb{ \rho_{\epsilon}}^2 + \normb{\nabla  c_{\epsilon}}^2+ \normb{ u_{\epsilon}}^2 + \normb{ \omega_{\epsilon}}^2 \right)
\end{split}
\end{equation*}
for $C=C\left(\norm{\rho_0}_{L^2\cap L^\infty},\norm{c_0}_{L^2\cap L^\infty},T^*\right)$.
Using the Gr\"{o}nwall's inequality, we obtain
\begin{equation*}
    \begin{split}
        \sup_{t\in[0,T^*]}&\left(\normb{ \rho_{\epsilon}}^2 + \normb{\nabla  c_{\epsilon}}^2+ \normb{ u_{\epsilon}}^2 + \normb{ \omega_{\epsilon}}^2 \right)(t) + \int_0^{T^*} \left(\normb{\nabla  \rho_{\epsilon}}^2 + \normb{\nabla^2  c_{\epsilon}}^2 \right)(s)ds \\
    &\qquad\le C\exp\left(\int_{0}^{T^*}\left(1+\normb{\nabla  c_{\epsilon}}^2\right)(s)\,ds \right)\left(\normb{\rho_{0,\epsilon}}^2 + \normb{\nabla c_{0,\epsilon}}^2+ \normb{u_{0,\epsilon}}^2 + \normb{\omega_{0,\epsilon}}^2 \right) \\
     &\qquad\le C\exp\left(\int_{0}^{T^*}\left(1+\normb{\nabla  c_{\epsilon}}^2\right)(s)\,ds \right)\left(\normb{\rho_0}^2 + \normb{\nabla c_0}^2+ \normb{u_0}^2 + \normb{\omega_0}^2 \right).
    \end{split}
\end{equation*}
Hence \eqref{est: c-1-2-2} and $\normb{\nabla  u_{\epsilon}} \lesssim \normb{ \omega_{\epsilon}}$  yield \eqref{est: L^2} and \eqref{est: nabla^2 c L^2} .
Recalling \eqref{ine: sob-1}, \eqref{est: u-L^4} follows from 
\begin{equation*}
    \norm{ u_{\epsilon}}_{L^4} \lesssim \normb{ u_{\epsilon}}^{\frac12}\normb{\nabla  u_{\epsilon}}^{\frac12} \lesssim \normb{ u_{\epsilon}}^{\frac12}\normb{ \omega_{\epsilon}}^{\frac12}. 
\end{equation*}
\end{proof}

With the above lemma at hand, we present $L^q$-estimate of $(\nabla  \rho_{\epsilon}, \nabla^2  c_{\epsilon}, \nabla  u_{\epsilon})$ with $q>2$ and $L^\infty$-estimate of $\nabla  c_{\epsilon}$.
\begin{lemma}
If $\normif{c_0}\le \delta_0$, then
\begin{subequations}
\begin{align}
    \sup_{t\in[0,T^*]}\left( \normq{\nabla \rho_{\epsilon}(t)}+\normq{\nabla^2  c_{\epsilon}(t)}+\normq{\nabla  u_{\epsilon}(t)}\right) &\le C
    \quad \text{for}\;\,q\in(2,\infty), \label{est: q}\\
    \sup_{t\in[0,T^*]}\normif{\nabla  c_{\epsilon}} &\le C \label{est: infty-1}
\end{align}
\end{subequations}
for a constant $C=C(q,\norm{\rho_0}_{L^2\cap W^{1,q}},\norm{c_0}_{H^1 \cap W^{2,q}},\norm{u_0}_{H^1\cap W^{1,q}},T^*)$.
\end{lemma}
\begin{proof}
We firstly consider the $\nabla  \rho_{\epsilon}$ equation:
\begin{equation*}
    \partial_{t} \nabla  \rho_{\epsilon} - \nabla \Delta  \rho_{\epsilon} = - u_{\epsilon}\cdot \nabla^2  \rho_{\epsilon} - \nabla  u_{\epsilon} \cdot \nabla  \rho_{\epsilon}
    -\nabla \left ( \nabla \cdot( \rho_{\epsilon} \nabla  c_{\epsilon})\right  ).
\end{equation*}
Taking the $L^2$ inner product of this equations with $\nabla  \rho_{\epsilon} |\nabla  \rho_{\epsilon}|^{q-2}$, we have
\begin{align*}
    \frac{1}{q} \frac{d}{dt} & \normq{\nabla  \rho_{\epsilon}}^q + \int |\nabla^2  \rho_{\epsilon}|^2|\nabla  \rho_{\epsilon}|^{q-2} + \frac{q-2}{4} \int |\nabla|\nabla  \rho_{\epsilon}|^2|^2|\nabla  \rho_{\epsilon}|^{q-4} \\
    &= - \int \left ( u_{\epsilon} \cdot \nabla^2  \rho_{\epsilon}\right  ) \cdot \nabla  \rho_{\epsilon} |\nabla  \rho_{\epsilon}|^{q-2} - \int \left ( \nabla  u_{\epsilon} \cdot \nabla  \rho_{\epsilon}\right  ) \cdot \nabla  \rho_{\epsilon} |\nabla  \rho_{\epsilon}|^{q-2} \\
    &\quad - \int \nabla \left ( \nabla \cdot\left ( \rho_{\epsilon} \nabla  c_{\epsilon}\right  )\right  ) \cdot \nabla  \rho_{\epsilon} |\nabla  \rho_{\epsilon}|^{q-2} \\
    &=: \RN{1}_1 + \RN{1}_2 +\RN{1}_3.
\end{align*}
Using the integration by parts, $\nabla \cdot  u_{\epsilon}=0$, and Young's inequality, we obtain
\begin{equation*}
    \begin{split}
    \RN{1}_1  &= 0,\\
    \RN{1}_2 &= \int  u_{\epsilon} \cdot \nabla  \rho_{\epsilon} \Delta  \rho_{\epsilon} |\nabla  \rho_{\epsilon}|^{q-2} 
    + \left (q-2\right  ) \int \left ( u_{\epsilon} \cdot \nabla  \rho_{\epsilon}\right  )\left (\nabla  \rho_{\epsilon} \cdot \nabla^2  \rho_{\epsilon} \cdot \nabla  \rho_{\epsilon}\right  )|\nabla  \rho_{\epsilon}|^{q-4} \\
    &\le (q-1) \int | u_{\epsilon}||\nabla^2  \rho_{\epsilon}||\nabla  \rho_{\epsilon}|^{q-1} \\
    &\le C(q) \int | u_{\epsilon}|^2|\nabla  \rho_{\epsilon}|^q + \frac{1}{4} \int |\nabla^2  \rho_{\epsilon}|^2|\nabla  \rho_{\epsilon}|^{q-2} \\
    &\le C(q) \normif{ u_{\epsilon}}^2 \normq{\nabla  \rho_{\epsilon}}^q + \frac{1}{4} \int |\nabla^2  \rho_{\epsilon}|^2|\nabla  \rho_{\epsilon}|^{q-2},\\
    \RN{1}_3 &= \int \nabla \cdot \left ( \rho_{\epsilon} \nabla  c_{\epsilon}\right  ) \cdot \Delta  \rho_{\epsilon} |\nabla  \rho_{\epsilon}|^{q-2} 
    + \left (q-2\right  ) \int \nabla \cdot \left ( \rho_{\epsilon} \nabla  c_{\epsilon}\right  )\left (\nabla  \rho_{\epsilon} \cdot \nabla^2  \rho_{\epsilon} \cdot \nabla  \rho_{\epsilon}\right  )|\nabla  \rho_{\epsilon}|^{q-4} \\
    &\le \left (q-1\right)  \left(\int |\nabla^2  \rho_{\epsilon}||\nabla  c_{\epsilon}||\nabla  \rho_{\epsilon}|^{q-1} 
    + \int | \rho_{\epsilon}| |\nabla^2  \rho_{\epsilon}||\nabla^2  c_{\epsilon}||\nabla  \rho_{\epsilon}|^{q-2} \right) \\
    &\le  C(q) \left (\int |\nabla  c_{\epsilon}|^2|\nabla  \rho_{\epsilon}|^q + \int | \rho_{\epsilon}|^2|\nabla^2  c_{\epsilon}|^2|\nabla  \rho_{\epsilon}|^{q-2} \right  )+ \frac{1}{2} \int |\nabla^2  \rho_{\epsilon}|^2|\nabla  \rho_{\epsilon}|^{q-2}\\
    &\le C(q)\left (\normif{\nabla  c_{\epsilon}}^2 \normq{\nabla  \rho_{\epsilon}}^q + \normif{ \rho_{\epsilon}}^2\normq{\nabla^2  c_{\epsilon}}^2\normq{\nabla  \rho_{\epsilon}}^{q-2}\right  )+ \frac{1}{2} \int |\nabla^2  \rho_{\epsilon}|^2|\nabla  \rho_{\epsilon}|^{q-2}\\
    &\le C(q)\left(\normif{ \rho_{\epsilon}}^2 + \normif{\nabla  c_{\epsilon}}^2\right) \left(\normq{\nabla  \rho_{\epsilon}}^q + \normq{\nabla^2  c_{\epsilon}}^q \right)  + \frac{1}{2} \int |\nabla^2  \rho_{\epsilon}|^2|\nabla  \rho_{\epsilon}|^{q-2}.
    \end{split}
\end{equation*}
Thus, using \eqref{est: rho-infty}, we obtain
\begin{equation}\label{est: nb rho q}
   \frac{d}{dt}\normq{\nabla  \rho_{\epsilon}}^q 
    \le C \left (1 + \normif{\nabla  c_{\epsilon}}^2 + \normif{ u_{\epsilon}}^2\right) \left (\normq{\nabla  \rho_{\epsilon}}^q + \normq{\nabla^2  c_{\epsilon}}^q  \right)
\end{equation}
for $C=C\left(q,\norm{\rho_0}_{L^2\cap L^\infty},\norm{c_0}_{L^2\cap L^\infty},T^*\right).$

Next, taking the $L^2$ inner product of \eqref{eq: nb c} with $\nabla  c_{\epsilon} |\nabla  c_{\epsilon}|^{q-2}$ and similarly proceeding as in the above estimate, we have
\begin{align*}
    \frac{1}{q}\frac{d}{dt} & \normq{\nabla  c_{\epsilon}}^q + \int |\nabla^2  c_{\epsilon}|^2|\nabla  c_{\epsilon}|^{q-2} + \frac{q-2}{4} \int |\nabla|\nabla  c_{\epsilon}|^2|^2|\nabla  c_{\epsilon}|^{q-4} \\
    &=-\int\left(\nabla  u_{\epsilon} \cdot \nabla  c_{\epsilon} \right)\cdot\nabla  c_{\epsilon} |\nabla  c_{\epsilon}|^{q-2} - \int \nabla \left ( c_{\epsilon} \rho_{\epsilon}\right  ) \cdot \nabla  c_{\epsilon} |\nabla  c_{\epsilon}|^{q-2} \\
    &\le C(q) \normif{ u_{\epsilon}}^2 \normq{\nabla  c_{\epsilon}}^q + \frac{1}{4} \int |\nabla^2  c_{\epsilon}|^2|\nabla  c_{\epsilon}|^{q-2} \\
    &\qquad + \int  c_{\epsilon} \rho_{\epsilon} \Delta  c_{\epsilon} |\nabla  c_{\epsilon}|^{q-2} + \left (q-2\right  )\int  c_{\epsilon} \rho_{\epsilon} \left (\nabla  c_{\epsilon} \cdot \nabla^2  c_{\epsilon} \cdot \nabla  c_{\epsilon}\right  )|\nabla  c_{\epsilon}|^{q-4} \\
     &\le C(q) \normif{ u_{\epsilon}}^2 \normq{\nabla  c_{\epsilon}}^q + \frac{1}{2} \int |\nabla^2  c_{\epsilon}|^2|\nabla  c_{\epsilon}|^{q-2} + C(q) \int | c_{\epsilon}|^2| \rho_{\epsilon}|^2|\nabla  c_{\epsilon}|^{q-2} \\
     &\le C(q)\left (\normif{ u_{\epsilon}}^2 \normq{\nabla  c_{\epsilon}}^q + \normif{ \rho_{\epsilon}}^2\normq{ c_{\epsilon}}^2\normq{\nabla  c_{\epsilon}}^{q-2}\right  ) +\frac{1}{2} \int |\nabla^2  c_{\epsilon}|^2|\nabla  c_{\epsilon}|^{q-2} \\
     &\le C(q) \left(\left(\normif{ \rho_{\epsilon}}^2 + \normif{ u_{\epsilon}}^2\right)\normq{\nabla  c_{\epsilon}}^q + \normif{ \rho_{\epsilon}}^2\normq{ c_{\epsilon}}^q\right)+\frac{1}{2} \int |\nabla^2  c_{\epsilon}|^2|\nabla  c_{\epsilon}|^{q-2}.
\end{align*}
Thus, using \eqref{est: c} and \eqref{est: rho-infty}, we obtain
\begin{align}\label{est: nb c-q}
    \frac{d}{dt}\normq{\nabla  c_{\epsilon}}^q &\le C\left (1+\left (1+\normif{ u_{\epsilon}}^2\right  )\normq{\nabla  c_{\epsilon}}^q \right)
\end{align}
for $C=C\left(q,\norm{\rho_0}_{L^2 \cap L^\infty},\norm{c_0}_{L^2 \cap L^\infty},T^*\right).$

We now consider the $\nabla^2  c_{\epsilon}$ equation:
\begin{equation*}
    \partial_{t} \nabla^2  c_{\epsilon} - \nabla^2 \Delta  c_{\epsilon} = - \nabla^2 \left ( u_{\epsilon} \cdot \nabla  c_{\epsilon}\right  )
    -\nabla^2 \left (c \rho_{\epsilon}\right  ).
\end{equation*}
Taking the $L^2$ inner product of this equations with $\nabla^2  c_{\epsilon} |\nabla^2  c_{\epsilon}|^{q-2}$, we have
\begin{align*}
    \frac{1}{q}\frac{d}{dt} & \normq{\nabla^2  c_{\epsilon}}^q + \int |\nabla^3  c_{\epsilon}|^2|\nabla^2  c_{\epsilon}|^{q-2} + \frac{q-2}{4} \int |\nabla|\nabla^2  c_{\epsilon}|^2|^2|\nabla^2  c_{\epsilon}|^{q-4} \\
    &= - \int \nabla^2 \left ( u_{\epsilon} \cdot \nabla  c_{\epsilon}\right  ) \cdot \nabla^2  c_{\epsilon} |\nabla^2  c_{\epsilon}|^{q-2} - \int \nabla^2 \left ( c_{\epsilon} \rho_{\epsilon}\right  ) \cdot \nabla^2  c_{\epsilon} |\nabla^2  c_{\epsilon}|^{q-2} \\
    &=: \RN{1}_4 + \RN{1}_5.
\end{align*}
The integration by parts, $\nabla \cdot  u_{\epsilon}=0$, and Young's inequality again give us
\begin{align*}
    \RN{1}_4 &\le (q-1) \left ( \int |\nabla  u_{\epsilon}||\nabla  c_{\epsilon}||\nabla^3  c_{\epsilon}||\nabla^2  c_{\epsilon}|^{q-2} 
    + \int | u_{\epsilon}||\nabla^2  c_{\epsilon}|^{q-1} |\nabla^3  c_{\epsilon}|\right ) \\
    &\le C(q) \left ( \int |\nabla  u_{\epsilon}|^2|\nabla  c_{\epsilon}|^2|\nabla^2  c_{\epsilon}|^{q-2} 
    + \int | u_{\epsilon}|^2|\nabla^2  c_{\epsilon}|^{q} \right ) + \frac{1}{2} \int |\nabla^3  c_{\epsilon}|^2 |\nabla^2  c_{\epsilon}|^{q-2} \\
    &\le C(q) \left ( \normif{\nabla  c_{\epsilon}}^2 \normq{\nabla  u_{\epsilon}}^2 \normq{\nabla  c_{\epsilon}}^{q-2} +\normif{ u_{\epsilon}}^2\normq{\nabla^2  c_{\epsilon}}^q  \right ) + \frac{1}{2} \int |\nabla^3  c_{\epsilon}|^2 |\nabla^2  c_{\epsilon}|^{q-2} \\
    &\le C(q) \left(\normif{\nabla  c_{\epsilon}}^2 +\normif{ u_{\epsilon}}^2\right)\left(\normq{\nabla^2  c_{\epsilon}}^q + \normq{\nabla  u_{\epsilon}}^q  \right)+ \frac{1}{2} \int |\nabla^3  c_{\epsilon}|^2 |\nabla^2  c_{\epsilon}|^{q-2}, \\
    \RN{1}_5 &\le (q-1) \left(\int | \rho_{\epsilon}||\nabla  c_{\epsilon}||\nabla^3  c_{\epsilon}||\nabla^2  c_{\epsilon}|^{q-2} + \int | c_{\epsilon}||\nabla  \rho_{\epsilon} ||\nabla^3  c_{\epsilon}||\nabla^2  c_{\epsilon}|^{q-2}\right) \\
    &\le C(q) \left ( \int | \rho_{\epsilon}|^2|\nabla  c_{\epsilon}|^2|\nabla^2  c_{\epsilon}|^{q-2} + \int | c_{\epsilon}|^2|\nabla  \rho_{\epsilon}|^2|\nabla^2  c_{\epsilon}|^{q-2}\right ) + \frac{1}{2} \int |\nabla^3  c_{\epsilon}|^2|\nabla^2  c_{\epsilon}|^{q-2} \\
    &\le C(q) \left( \normif{ \rho_{\epsilon}}^2 \normq{\nabla  c_{\epsilon}}^2 \normq{\nabla^2  c_{\epsilon}}^{q-2} + \normif{ c_{\epsilon}}^2 \normq{\nabla  \rho_{\epsilon}}^2 \normq{\nabla^2  c_{\epsilon}}^{q-2} \right) +\frac{1}{2} \int |\nabla^3  c_{\epsilon}|^2|\nabla^2  c_{\epsilon}|^{q-2}\\
    &\le C(q)\left(\normif{ \rho_{\epsilon}}^2 + \normif{ c_{\epsilon}}^2 \right)\left (\normq{\nabla  \rho_{\epsilon}}^q +\normq{\nabla  c_{\epsilon}}^q + \normq{\nabla^2  c_{\epsilon}}^q  \right)+\frac{1}{2} \int |\nabla^3  c_{\epsilon}|^2|\nabla^2  c_{\epsilon}|^{q-2}.
\end{align*}
Thus, using \eqref{est: c} and \eqref{est: rho-infty}, we obtain
\begin{equation}\label{est: nb^2 c-q}
    \frac{d}{dt}\normq{\nabla^2  c_{\epsilon}}^q 
    \le C \left (1 + \normif{\nabla  c_{\epsilon}}^2+\normif{ u_{\epsilon}}^2\right)\left (\normq{\nabla  u_{\epsilon}}^q + \normq{\nabla  \rho_{\epsilon}}^q +\normq{\nabla  c_{\epsilon}}^q + \normq{\nabla^2  c_{\epsilon}}^q  \right)
\end{equation}
for $C=C\left(q,\norm{\rho_0}_{L^2\cap L^\infty},\norm{c_0}_{L^2\cap L^\infty},T^*\right).$
After taking $L^2$ inner product of the vorticity equation \eqref{eq: vorticity} with $ \omega_{\epsilon} | \omega_{\epsilon}|^{q-2}$, we obtain
\begin{equation}\label{est: omega-q}
    \frac{d}{dt}\normq{ \omega_{\epsilon}}^q + \epsilon \int |\nabla \omega_\epsilon|^2|\omega_\epsilon|^{q-2} + \frac{\epsilon(q-2)}{4}\int |\nabla |\omega_\epsilon|^2 |^2 |\omega_\epsilon|^{q-4}\le C \left(\normq{ \omega_{\epsilon}}^q + \normq{\nabla  \rho_{\epsilon}}^q\right)
\end{equation}
for $C=C\left(q\right).$
Summing from \eqref{est: nb rho q} to \eqref{est: omega-q} and recalling the Calder\'on-Zygmund inequality: $\normq{\nabla  u_{\epsilon}}\lesssim \normq{ \omega_{\epsilon}}$, we have 
\begin{equation*}
    \frac{d}{dt}Y_\epsilon \le C\left (\normif{\nabla  c_{\epsilon}}^2+\normif{ u_{\epsilon}}^2\right)Y_\epsilon,
\end{equation*}
for $C=C\left(q,\norm{\rho_0}_{L^2\cap  L^\infty},\norm{c_0}_{L^2 \cap L^\infty},T^*\right),$
where $Y_\epsilon$ is defined by
\begin{equation*}
    Y_\epsilon(t):=\left( \normq{\nabla  \rho_{\epsilon}}^q +\normq{\nabla  c_{\epsilon}}^q + \normq{\nabla^2  c_{\epsilon}}^q+\normq{ \omega_{\epsilon}}^q \right)(t).
\end{equation*}
Recalling the Br\'ezis-Wainger inequality \eqref{ine: B-W},
we estimate
\begin{align*}
   \normif{\nabla  c_{\epsilon}}^2+ \normif{ u_{\epsilon}}^2 
    &\lesssim\left( 1 + \normb{\nabla  u_{\epsilon}} + \normb{\nabla^2  c_{\epsilon}}\right)^2 \left( 1+\log_+\left(\normq{\nabla^2  c_{\epsilon}} + \normq{\nabla  u_{\epsilon}} \right)\right)  + \left(\normb{ u_{\epsilon}} + \normb{\nabla  c_{\epsilon}}\right)^2 \\
    &\le C \left(1+ \normb{\nabla^2  c_{\epsilon}}^2\right)\left( 1+ \log_+\left(\normq{\nabla^2  c_{\epsilon}} + \normq{ \omega_{\epsilon}} \right)\right)
\end{align*}
for a constant $C=C(\norm{\rho_0}_{L^2\cap L^\infty},\norm{c_0}_{L^\infty \cap H^1},\norm{u_0}_{H^1},T^*)$,
where we have used \eqref{est: L^2} in the last line.
Therefore, we arrive at
\begin{equation*}
    \frac{d}{dt} Y_\epsilon \le C\left(1+ \normb{\nabla^2  c_{\epsilon}}^2\right) \left(1 +\log (1+Y_\epsilon)\right) Y_\epsilon,
\end{equation*}
which yields
\begin{equation*}
    \frac{d}{dt} \left(1 +\log (1+Y_\epsilon)\right) \le C\left(1+ \normb{\nabla^2  c_{\epsilon}}^2\right) \left(1 +\log (1+Y_\epsilon)\right).
\end{equation*}
for a constant $C=C(q,\norm{\rho_0}_{L^2\cap L^\infty},\norm{c_0}_{L^\infty \cap H^1},\norm{u_0}_{H^1},T^*)$.
Applying the Gr\"{o}nwall's inequality and recalling \eqref{est: nabla^2 c L^2}, $\normq{\nabla  u_{\epsilon}}\lesssim \normq{ \omega_{\epsilon}}$, and \begin{equation*}
    Y_\epsilon(0) \le C \left( \normq{\nabla \rho_{0}}^q +\normq{\nabla c_{0}}^q + \normq{\nabla^2 c_{0}}^q+\normq{\omega_{0}}^q \right),
\end{equation*} we can obtain \eqref{est: q}.
\eqref{est: infty-1} follows from the Gagliardo-Nirenberg interpolation inequality:
\begin{equation*}
    \normifr{f} \lesssim\normqr{f}^{\frac{q-2}{2q-2}}\normqr{\nabla f}^{\frac{q}{2q-2}}\;\, \text{with} \;\, q>2.
\end{equation*}
\end{proof}

We finally provide $L^{\infty}$-estimate of $(\nabla  \rho_{\epsilon}, \nabla^2  c_{\epsilon},  \omega_{\epsilon})$. 

\begin{lemma}
If $\normif{c_0}\le \delta_0$, then 
\begin{equation}\label{est: W1,infty}
    \sup_{t\in[0,T^*]}\left( \normif{\nabla \rho_{\epsilon}(t)}+\normif{\nabla^2  c_{\epsilon}(t)}+\normif{ \omega_{\epsilon}(t)}\right) \le C
\end{equation}
for a constant $C=C(T_{loc},T^*,\norm{\rho_0}_{H^m},\norm{c_0}_{H^{m+1}},\norm{u_0}_{H^{m+1}})$ with $m>3$, where $T_{loc}$ is from \eqref{definition of Tloc}.  
\end{lemma}
\begin{proof}
We first show that there exists a constant $C=C(T_{loc},T^*,\norm{\rho_0}_{H^m},\norm{c_0}_{H^{m+1}},\norm{u_0}_{H^{m+1}})$ such that
\begin{equation}\label{est: nb rho, nb^c-infty}
   \normif{\nabla  \rho_{\epsilon}(t)} + \normif{\nabla^2  c_{\epsilon}(t)} \le C  \quad\text{for} \;\, t\in\left[\frac{T_{loc}}{2},T^*\right].
\end{equation}
To obtain \eqref{est: nb rho, nb^c-infty}, we consider integral forms of $ \rho_{\epsilon}$ and $c_\epsilon$ given by the Duhamel's principle: 
\begin{equation*}
    \begin{split}
    & \rho_{\epsilon}(t)=e^{t\Delta}\rho_{0,\epsilon} - \int_{0}^{t} \nabla e^{(t-s)\Delta} ( \rho_{\epsilon}(s)\nabla  c_{\epsilon}(s)) \; ds - \int_{0}^{t} \nabla e^{(t-s)\Delta} ( u_{\epsilon}(s) \rho_{\epsilon}(s)) \; ds , \\
    & c_{\epsilon}(t)=e^{t\Delta}c_{0,\epsilon} - \int_{0}^{t} e^{(t-s)\Delta} ( \rho_{\epsilon}(s)  c_{\epsilon}(s)) \; ds - \int_{0}^{t} \nabla e^{(t-s)\Delta} ( u_{\epsilon}(s)  c_{\epsilon}(s)) \; ds.
    \end{split}
\end{equation*}
We find the following heat kernel estimate useful:
\begin{equation}\label{est: 1}
    \normqr{\nabla^{\alpha} e^{\nu t\Delta}f} \lesssim\frac{\normrr{f}}{(\nu t)^{\frac{1}{r}-\frac{1}{q}+\frac{|\alpha|}{2}}} \quad \text{for} \;\, 1\le r \le q \le \infty.
\end{equation}
Note that our previous lemmas imply
\begin{equation*}
    \sup_{t\in [0,T^*]}\left( \normif{ \rho_{\epsilon}} + \normif{ c_{\epsilon}} + \normif{\nabla  c_{\epsilon}} + \normb{ u_{\epsilon}} + \normd{ u_{\epsilon}} + \normd{\nabla  u_{\epsilon}} \right)(t)\le M
\end{equation*}
for a constant $M=M(\norm{\rho_0}_{H^m},\norm{c_0}_{H^{m+1}},\norm{u_0}_{H^{m+1}},T^*)$.
Let $\delta \in (0, \frac{1}{2}]$ be a constant to be chosen later. 
In the following estimates, $C(M)$'s denote constants dependent on $M$ but independent of $\delta$.

For the estimate of $\normif{\nabla  \rho_{\epsilon}}$, we note that
\begin{align*}
    \normif{\nabla  \rho_{\epsilon}(t)} 
    &\le  \normif{\nabla e^{t\Delta}\rho_{0,\epsilon}} + \int_{0}^{(1-\delta)t} \normif{\nabla^2 e^{(t-s)\Delta}  \rho_{\epsilon}(s) \nabla  c_{\epsilon}(s))} ds \\
    &\quad + \int_{(1-\delta)t}^{t} \normif{\nabla e^{(t-s)\Delta} \nabla ( \rho_{\epsilon}(s)\nabla  c_{\epsilon}(s))} ds 
    + \int_{0}^{(1-\delta)t} \normif{\nabla^2 e^{(t-s)\Delta} ( u_{\epsilon}(s)  \rho_{\epsilon}(s))} ds \\
    &\quad + \int_{(1-\delta)t}^{t} \normif{\nabla e^{(t-s)\Delta} ( u_{\epsilon}(s) \nabla  \rho_{\epsilon}(s))}  ds \\
    &=: \RN{2}_1 +\RN{2}_2 +\RN{2}_3 +\RN{2}_4 +\RN{2}_5.
\end{align*}
It is clear from \eqref{est: 1} that
\begin{equation*}
    \RN{2}_1 \le \frac{C\normif{\rho_{0,\epsilon}}}{t^\frac{1}{2}} \le  \frac{C\normif{\rho_0}}{t^\frac{1}{2}}.
\end{equation*}
Moreover, we estimate
\begin{equation*}
    \begin{split}
        \RN{2}_2 &\le C \int_{0}^{(1-\delta)t} \frac{\normif{ \rho_{\epsilon}(s)}\normif{\nabla  c_{\epsilon}(s)}}{t-s}\,ds
    \le C(M)\int_{0}^{(1-\delta)t} \frac{1}{t-s}\,ds
    \le\frac{C(M)}{\delta t}, \\
    \RN{2}_3 &\le C \int_{(1-\delta)t}^{t} \frac{\normif{\nabla  \rho_{\epsilon}(s)}\normif{\nabla  c_{\epsilon}(s)} + \normif{ \rho_{\epsilon} (s)}\normif{\nabla^2  c_{\epsilon}(s)}}{(t-s)^\frac{1}{2}}\,ds \\
    & \le C(M) \sup_{\frac{t}{2}\le s \le t}\left( \normif{\nabla  \rho_{\epsilon}(s)} + \normif{\nabla^2  c_{\epsilon}(s)}\right)\int_{(1-\delta)t}^{t} \frac{1}{(t-s)^\frac{1}{2}}\,ds  \\
    & \le C(M) (\delta t)^{\frac{1}{2}}\sup_{\frac{t}{2}\le s \le t}\left( \normif{\nabla  \rho_{\epsilon}(s)} + \normif{\nabla^2  c_{\epsilon}(s)}\right), \\
    \RN{2}_4 &\le C \int_{0}^{(1-\delta)t} \frac{\normb{ u_{\epsilon}(s)}\normif{ \rho_{\epsilon}(s)}}{(t-s)^{\frac{3}{2}}}\,ds
    \le C(M) \int_{0}^{(1-\delta)t} \frac{1}{(t-s)^\frac{3}{2}}\,ds
    \le\frac{C(M)}{(\delta t)^{\frac{3}{2}}}, \\
    \RN{2}_5&\le C \int_{(1-\delta)t}^{t} \frac{\normd{ u_{\epsilon}(s)}\normif{\nabla  \rho_{\epsilon}(s)}}{(t-s)^\frac{3}{4}}\,ds  \le C(M) (\delta t)^{\frac{1}{4}} \sup_{\frac{t}{2}\le s \le t}  \normif{\nabla  \rho_{\epsilon}(s)}.
    \end{split}
\end{equation*}

Next, we note that
\begin{align*}
    \normif{\nabla^2  c_{\epsilon}(t)} 
    &\le  \normif{\nabla^2 e^{t\Delta} c_{0,\epsilon}} + \int_{0}^{(1-\delta)t} \normif{\nabla^2 e^{(t-s)\Delta} ( \rho_{\epsilon}(s)  c_{\epsilon}(s))} ds \\
    &\quad + \int_{(1-\delta)t}^{t} \normif{\nabla e^{(t-s)\Delta} \nabla ( \rho_{\epsilon}(s)  c_{\epsilon}(s))} ds 
    + \int_{0}^{(1-\delta)t} \normif{\nabla^2 e^{(t-s)\Delta} ( u_{\epsilon}(s) \nabla  c_{\epsilon}(s))} ds \\
    &\quad + \int_{(1-\delta)t}^{t} \normif{\nabla e^{(t-s)\Delta} \nabla ( u_{\epsilon}(s) \nabla  c_{\epsilon}(s))}  ds \\
    &=: \RN{3}_1 +\RN{3}_2 +\RN{3}_3 +\RN{3}_4 +\RN{3}_5.
\end{align*}
It is clear from \eqref{est: 1} that
\begin{equation*}
    \RN{3}_1 \le \frac{C\normif{c_{0,\epsilon}}}{t}\le\frac{C\normif{c_0}}{t}.
\end{equation*}
We further estimate
\begin{equation*}
    \begin{split}
        \RN{3}_2 &\le C \int_{0}^{(1-\delta)t} \frac{\normif{ c_{\epsilon}(s)}\normif{ \rho_{\epsilon}(s)}}{t-s}\,ds \le \frac{C(M)}{\delta t}, \\
         \RN{3}_3 & \le C \int_{(1-\delta)t}^{t} \frac{\normif{ \rho_{\epsilon}(s)}\normif{\nabla  c_{\epsilon}(s)} + \normif{\nabla  \rho_{\epsilon} (s)}\normif{ c_{\epsilon}(s)}}{(t-s)^\frac{1}{2}}\,ds \\
    & \le C(M) \int_{(1-\delta)t}^{t} \frac{1 + \normif{\nabla  \rho_{\epsilon} (s)}}{(t-s)^\frac{1}{2}}\,ds \\
    & \le C(M) (\delta t)^{\frac{1}{2}} \left( 1 + \sup_{\frac{t}{2}\le s \le t} \normif{\nabla  \rho_{\epsilon}(s)}\right), \\
    \RN{3}_4 &\le C \int_{0}^{(1-\delta)t} \frac{\normb{ u_{\epsilon}(s)}\normif{\nabla  c_{\epsilon}(s)}}{(t-s)^{\frac{3}{2}}}
    \le \frac{C(M)}{(\delta t)^{\frac{3}{2}}},\\
    \RN{3}_5&\le C \int_{(1-\delta)t}^{t} \frac{\normd{\nabla  u_{\epsilon}(s)}\normif{\nabla  c_{\epsilon}(s)} + \normd{ u_{\epsilon}(s)}\normif{\nabla^2  c_{\epsilon}(s)}}{(t-s)^\frac{3}{4}}\,ds \\
    &\le C(M) (\delta t)^{\frac{1}{4}}  \left( 1 + \sup_{\frac{t}{2}\le s \le t}  \normif{\nabla^2  c_{\epsilon}(s)}\right).
    \end{split}
\end{equation*}

Combining all and noticing
\begin{equation*}
    \sup_{\frac{t}{2}\le s \le t}\left( \normif{\nabla  \rho_{\epsilon}(s)} + \normif{\nabla^2  c_{\epsilon}(s)}\right)
    \le \sup_{\frac{T_{loc}}{2}\le s < T^*}\left( \normif{\nabla  \rho_{\epsilon}(s)} + \normif{\nabla^2  c_{\epsilon}(s)}\right)
\end{equation*}
for $t\in[T_{loc},T^*)$,
we have
\begin{equation*}
    \normif{\nabla  \rho_{\epsilon}(t)}+\normif{\nabla^2  c_{\epsilon}(t)}\le C(M)\left((\delta T_{loc})^{-\frac{3}{2}}  + (\delta T^*)^{\frac14}    \sup_{\frac{T_{loc}}{2}\le s < T^*}\left( \normif{\nabla  \rho_{\epsilon}(s)} + \normif{\nabla^2  c_{\epsilon}(s)}\right)\right)
\end{equation*}
for $t\in[T_{loc},T^*)$.
Choosing 
\begin{equation*}
    \delta= \min \left\{\frac12, \,\frac{1}{16C(M)^4 T^*} \right\},
\end{equation*}  
we arrive at
\begin{equation*}
\begin{split}
    \sup_{\frac{T_{loc}}{2}\le s < T^*}\left( \normif{\nabla  \rho_{\epsilon}(s)} + \normif{\nabla^2  c_{\epsilon}(s)}\right)
    &\le \sup_{\frac{T_{loc}}{2}\le s < T_{loc}}\left( \normif{\nabla  \rho_{\epsilon}(s)} + \normif{\nabla^2  c_{\epsilon}(s)}\right) \\
    &\quad+  C(M,T_{loc},T^*)  + \frac{1}{2}\sup_{\frac{T_{loc}}{2}\le s < T^*}\left( \normif{\nabla  \rho_{\epsilon}(s)} + \normif{\nabla^2  c_{\epsilon}(s)}\right).
\end{split}
\end{equation*}
As we observed in \eqref{bound of X}, we have
\begin{equation*}
    \sup_{\frac{T_{loc}}{2}\le s < T_{loc}}\left( \normif{\nabla  \rho_{\epsilon}(s)} + \normif{\nabla^2  c_{\epsilon}(s)}\right) \le C=C(T_{loc},\norm{\rho_0}_{H^m},\norm{c_0}_{H^{m+1}},\norm{u_0}_{H^{m+1}}),
\end{equation*} 
which implies \eqref{est: nb rho, nb^c-infty}.

Next, along the particle-trajectory map $\Psi_\epsilon\left (\cdot,t\right):\mathbb{R}^2 \rightarrow \mathbb{R}^2$ (\cite{Majda}), the vorticity equation \eqref{eq: vorticity} gives
\begin{equation*}
    \partial_{t}  \omega_{\epsilon} \left(\Psi_\epsilon(x,t),t\right) = \epsilon \Delta \omega_\epsilon  \left(\Psi_\epsilon(x,t),t\right)+\left(\nabla^{\bot}  \rho_{\epsilon} \cdot \nabla \phi \right)\left(\Psi_\epsilon(x,t),t\right),
\end{equation*}
and therefore, the bijectivity of $\Psi_\epsilon\left(\cdot,t\right)$ and  the Duhamel principle and 
\eqref{est: nb rho, nb^c-infty} imply that
\begin{equation*}
    \normif{ \omega_{\epsilon}(t)} \le \normif{e^{\epsilon t\Delta}\omega_{0,\epsilon}}+\int_{0}^{t}  \normif{e^{\epsilon t\Delta}(\nabla \rho_{\epsilon}(s)\cdot\nabla \phi)}ds
\end{equation*}
on $[0,T^*]$. 
Hence we use \eqref{est: nb rho, nb^c-infty} and \eqref{est: 1} with $\normif{\omega_{0,\epsilon}}\le C\normif{\omega_0}$ to obtain an upper bound $C=C(T_{loc},T^*,\norm{\rho_0}_{H^m},\norm{c_0}_{H^{m+1}},\norm{u_0}_{H^{m+1}})$ of $\sup_{t\in[0,T^*]}\normif{\omega_\epsilon(t)}$.
\end{proof}
Now we are ready to show Theorem \ref{thm: gwp}.
Considering \eqref{est: c}, \eqref{est: rho-infty}, \eqref{est: q}, and \eqref{est: infty-1}, 
\eqref{est: local} becomes
\begin{equation*}
    \frac{d}{dt}X_\epsilon+\normh{\nabla  \rho_{\epsilon}}^2 + \normhh{\nabla  c_{\epsilon}}^2
        \le C \left (1+\normif{\nabla  u_{\epsilon}}\right) X_\epsilon
\end{equation*}
for a constant $C=C(T_{loc},T^*,\norm{\rho_0}_{H^m},\norm{c_0}_{H^{m+1}},\norm{u_0}_{H^{m+1}})$.
Plugging into this an elementary estimate in $\mathbb{R}^2$:
\begin{align*}
    \normif{\nabla  u_{\epsilon}} &\lesssim \left( \normb{ u_{\epsilon}} + \normif{ \omega_{\epsilon}} \left( 1+ \log_{+}  \frac{\norm{ \omega_{\epsilon}}_{H^2}}{\normif{ \omega_{\epsilon}}} \right)  \right),
\end{align*}
\eqref{est: L^2} and \eqref{est: W1,infty} give us
\begin{equation}\label{final estimate}
    \frac{d}{dt}X_\epsilon + \normh{\nabla  \rho_{\epsilon}}^2 + \normhh{\nabla  c_{\epsilon}}^2 \le C\left(1+\log(1+X_\epsilon) \right)X_\epsilon
\end{equation}
for  $C=C(T_{loc},T^*,\norm{\rho_0}_{H^m},\norm{c_0}_{H^{m+1}},\norm{u_0}_{H^{m+1}})$.
Using the Gr\"{o}nwall's inequality and noticing
\begin{equation*}
    X_\epsilon(0) \le C \left( \normh{\rho_{0}}^2 + \normhh{c_{0}}^2 + \normhh{u_{0}}^2\right),
\end{equation*} we conclude that $(\rho_{\epsilon},c_{\epsilon},u_{\epsilon})$ and $(\nabla\rho_{\epsilon},\nabla c_{\epsilon})$ are uniformly-in-$\epsilon$ bounded in 
\begin{equation*}
    C \left ([0,T^{*}]; H^{m} \times H^{m+1} \times H^{m+1}\right) \quad \text{and} \quad L^2 \left ([0,T^{*}]; H^{m} \times H^{m+1} \right),
\end{equation*}
respectively.
Now proceeding in the same way as Step 1 in the previous section, we can show the existence of a convergent subsequence $(\rho_{\epsilon_{n}},c_{\epsilon_{n}},u_{\epsilon_{n}})$ whose limit $(\rho,c,u)$ as $\epsilon_n\rightarrow 0$ is a solution of \eqref{KS-E} with \eqref{intro} satisfying
\begin{equation*}
    (\rho,c,u)\in C \left ([0,T^{*}]; H^{m} \times H^{m+1} \times H^{m+1}\right), \quad (\nabla \rho,\nabla c) \in L^2 \left ([0,T^{*}]; H^{m} \times H^{m+1} \right).
\end{equation*}
Since $T^*>0$ is arbitary, we have shown the global existence of solutions.
In order to establish the uniqueness of the solution, we can employ the same argument as in Step 2 from the previous section.
This completes the proof of Theorem \ref{thm: gwp}. $\Box$

\subsection*{Acknowledgments}{The author was supported by the Samsung Science and Technology Foundation under Project Number SSTF-BA2002-04. He thanks In-Jee Jeong for educational discussions and comments.}

\bibliographystyle{amsplain}
\bibliography{Keller-Segel-fluid}

\end{document}